\documentclass[10pt]{article}
\usepackage{amssymb}
\usepackage{amsmath}
\usepackage{amsthm}
\usepackage{epsfig}
\usepackage{mathrsfs}
\usepackage{xcolor}
\usepackage{graphicx}
\usepackage{color}

\newtheorem{theorem}{Theorem}[section]
\newtheorem{lemma}[theorem]{Lemma}
\newtheorem{proposition}[theorem]{Proposition}
\newtheorem{definition}[theorem]{Definition}

\newtheorem{remark}[theorem]{Remarks}
\newtheorem{rk&ex}[theorem]{Remarks \& Examples}
\newtheorem{corollary}[theorem]{Corollary}

\def\NN{{\mathbb N}}
\def\ZZ{{\mathbb Z}}
\def\Z{{\mathbb Z}}
\def\RR{{\mathbb R}}
\def\R{{\mathbb R}}

\def\e{\varepsilon}

\begin{document}
\title{Homogenization for nonlinear PDEs in general domains with oscillatory Neumann boundary data}
\author{Sunhi Choi \thanks{Department of Mathematics, U. of Arizona, Tucson, AZ. }  and Inwon C. Kim
\thanks{Department of Mathematics, UCLA, LA CA. Research supported by NSF DMS 0970072} }
\date{}
\maketitle
\begin{abstract}
In this article we investigate averaging properties of fully nonlinear PDEs in bounded domains with oscillatory Neumann boundary data. The oscillation is periodic and is present both in the operator and in the Neumann data. Our main result states that, when the domain does not have flat boundary parts and when the homogenized operator is rotation invariant, the solutions uniformly converge to the homogenized solution solving a Neumann boundary problem. Furthermore we show that the homogenized Neumann data is continuous with respect to the normal direction of the boundary.  Our result is the nonlinear version of the classical result in \cite{BLP} for divergence-form operators with co-normal boundary data. The main ingredients in our analysis are the estimate on the oscillation on the solutions in half-spaces (Theorem~\ref{thm:planar}), and the estimate on the mode of convergence of the solutions as the normal of the half-space varies over {\it irrational} directions (Theorem~\ref{continuity}).
\end{abstract}

\section{Introduction}

Let us consider a bounded domain $\Omega$ in $\R^n$ containing the closed unit ball $K=\{x:|x|\leq 1\}$ (see Figure 1). Let $g:\R^n\to [1,2] $ be a H\"{o}lder continuous function which is periodic with respect to the orthonomal basis $\{(e_1,...,e_n)\}$ of $\R^n$. More precisely, $g$ satisfies
$$
g(x+e_i)=g(x) \quad\hbox{ for } i=1,...,n \hbox{ and } \quad g\in C^{\beta}(\R^n) \hbox{ for some } 0<\beta<1.
$$

With $\Omega, K$ and $g$ as given above, we are interested in the limiting behavior of the following problem:
$$
\left\{\begin{array}{lll}
F(D^2 u^\e,\frac{x}{\e})=0 &\hbox{ in } \Omega-K,\\ \\
u^\e=1 &\hbox{ on } K,\\ \\
\frac{\partial}{\partial\nu} u^\e = g(\frac{x}{\e}) &\hbox{ on  } \partial\Omega.
\end{array}\right. \leqno(P_\e)
$$
 \begin{figure}
\center{\epsfig{file=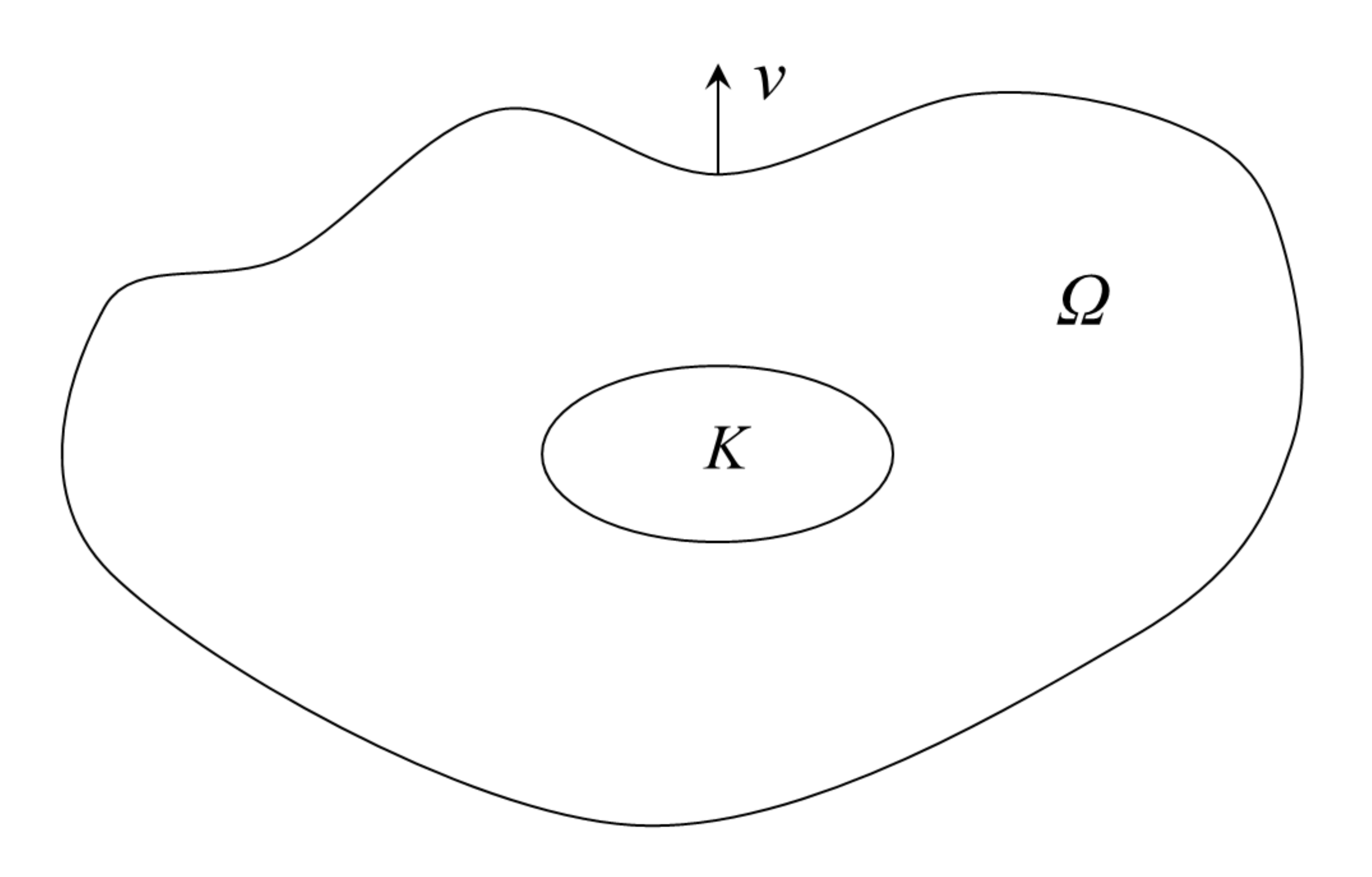,height=1.5in}} \center{Figure 1}
\end{figure}

Here $\nu=\nu_x$ is the outward normal vector at $x\in\partial\Omega$ and $F$ is a uniformly elliptic, fully nonlinear operator (see section 2 for the precise definitions and conditions on $F$). 
The set $K$ is introduced to avoid discussions of compatibility conditions on the Neumann boundary data. The operators under discussion include, for example, the divergence or non-divergence form operator

\begin{equation}\label{example}
-\Sigma_{i,j} a_{ij}(\frac{x}{\e}) \partial_{x_ix_j} u=0,
\end{equation}
where $a_{ij}(x)$ are H\"{o}lder continuous  and $\lambda I_{n\times n} \leq (a_{ij})(x)\leq \Lambda I_{n\times n}$ for positive constants $\lambda$ and $\Lambda$.

\medskip

Let us define
$$
\mathcal{S}^{n-1}=\{\nu\in\R^n:|\nu|=1\}\quad\hbox{ and }\quad\mathcal{M}^n:=\{ M: n\times n \hbox{ symmetric matrix }\}.
$$

Due to \cite{CS}, there exists a uniformly elliptic operator $\bar{F}: \mathcal{M}^n\to \R$ such that any solution of $F(D^2u^\e,\frac{x}{\e})=0$ in $\Omega$ with {\it fixed} boundary data on $\partial\Omega$ converges uniformly to the solution of $\bar{F}(D^2u)=0$ in $\Omega$ (see Theorem~\ref{original_convergence}). Therefore the question under investigation is whether the oscillatory boundary data changes the averaging behavior of $u^\e$. As we will discuss below, several difficulties arise in answering this question due to the nonlinear (or non-divergence form) nature of the problem as well as the geometry of the domain.

\medskip

 Let us state our main result.  First we introduce the following ``cell problem" for given $\nu\in\mathcal{S}^{n-1}$ and $\lambda\in\R$.
 
 $$
\left\{\begin{array}{lll}
F(D^2 u^\e,\frac{x}{\e})=0 &\hbox{ in } \{-1\leq x\cdot\nu -\lambda \leq 0\};\\ \\
u^\e=1 &\hbox{ on } \Gamma_D:= \{x\cdot\nu=\lambda-1\};\\ \\
\partial_\nu u^\e = g(\frac{x}{\e}) &\hbox{ on  }\Gamma_N:=\{x\cdot\nu = \lambda\}.
\end{array}\right.\leqno (P)_{\e,\nu,\lambda}
$$

\begin{theorem}\label{main} 
Let $\bar{F}:\mathcal{M}^n\to\R$ and, $(P)_\e$ and $(P)_{\e,\nu,\lambda, a}$ be as given above.
\begin{itemize}
\item[(a)] [Theorem 3.1] For given normal direction $\nu\in\mathcal{S}^{n-1}-\R\Z^n$ and any $\lambda\in\R$ there exists a unique homogenized neumann data $\mu(\nu)$ for solutions of $(P_{\e,\nu,\lambda})$. 
\item[(b)][Theorem 5.1] Suppose that $\bar{F}(M)$ only depends on the eigenvalues of $M\in\mathcal{M}^n$. Then there exists a continuous function $\bar{\mu}(\nu): \mathcal{S}^{n-1}\to \R$, given as the continuous extension of $\mu(\nu)$ over $\nu \in \mathcal{S}^{n-1} - \R\Z^n$,  such that the following holds:

Suppose $\Omega$ is a bounded domain in $\R^n$ such that $\partial\Omega$ is $C^2$ and  does not contain any flat part. Let $u^\e$ solve $(P)_\e$.  Then $u^\e$ converges locally uniformly to  $u$, which is the unique solution of the homogenized problem
$$
\left\{\begin{array}{lll}
\bar{F}(D^2 u)=0 &\hbox{ in } \Omega-K;\\ \\
u=1 &\hbox{ on } K;\\ \\
\partial_\nu u = \bar{\mu}(\nu) &\hbox{ on  } \partial\Omega.
\end{array}\right. \leqno(\bar{P})
$$
\end{itemize}
\end{theorem}

\begin{remark}
\begin{itemize}
\item[1.]  The assumption on $\bar{F}$ in Theorem~\ref{main} (b) seems to be necessary condition to achieve the continuity of the homogenized slope $\bar{\mu}(\nu)$, at least from numerical experiments (see \cite{FS}). Note that this assumption is equivalent to the rotation and reflection invariance of $\bar{F}$. For more discussion on the role of the assumption in the analysis, see the discussion in section 4.\\

\item[2.] We point out that the restriction on the geometry of $\partial\Omega$ is necessary: in fact, if  $\partial\Omega$ is locally $\Gamma=\{x\cdot\nu = \lambda\}$ with $\nu\in\ZZ^n$ and $\lambda\neq 0$, then the (normalized) distribution of boundary data $g(\frac{x}{\e})$ on $\Gamma$ changes a lot as $\e\to 0$.  Consequently in this case there is no unique limit of $u^\e$.
\end{itemize}
\end{remark}

\vspace{15pt}

$\circ${\it Discussions on previous results}

\medskip

Our problem is classical for the case of uniformly elliptic, divergence-form equations
\begin{equation}\label{div}
-\nabla\cdot(A(\frac{x}{\e}) \nabla u^\e)=0 \hbox{ in } \Omega,
\end{equation}
 with  the {\it co-normal} boundary data
 \begin{equation}\label{conormal}
 \nu\cdot(A(\frac{x}{\e})\nabla u^\e)(x)=g(\frac{x}{\e}) \hbox{ on } \partial\Omega.
 \end{equation} 

For \eqref{div}-\eqref{conormal},  a corresponding result to Theorem~\ref{main} was proved in the classical paper of Bensoussan, Lions and Papanicolau \cite{BLP} with explicit integral formula for the limiting operator as well as the limiting boundary data. Recently, a corresponding result was shown for systems of divergence type operators with oscillatory Dirichlet boundary data in convex domains, by Gerard-Varet and Masmoudi \cite{GM} (also see \cite{LS}).

\medskip

For nonlinear or non-divergence type operators, or even for linear operators with oscillatory Neumann boundary data that is not co-normal, most available results concern half-space type domains whose boundary goes through the origin. In \cite{T}, Tanaka considered some model problems in half-space whose boundary is parallel to the axes of the periodicity by purely probabilistic methods. In \cite{Ari} Arisawa studied special cases of problems in oscillatory domains near half spaces going through the origin, using viscosity solutions as well as stochastic control theory. Generalizing the results of \cite{Ari}, Barles, Da Lio and Souganidis \cite{BDLS} studied the problem for operators with oscillating coefficients, in half-space type domains whose boundary is parallel to the axes of periodicity, with a series of assumptions which guarantee the existence of approximate corrector. In \cite{CKL} the continuity property of the averaged Neumann boundary data in half-spaces, with respect to the normal direction, was studied in the case of homogeneous operator $F$. Recently, Barles and Mironescu \cite{BM} showed a corresponding result to \cite{BDLS} for oscillatory Dirichlet boundary data. 

\vspace{15pt}

$\circ$ {\it Main ingredients and challenges}

\medskip

The main steps in extending aforementioned results from half-space type domains to general domains are the following.

The plan is to use these solutions in strip domains to approximate those in general domains. For the stability of such approximation one requires the distribution of $g$ to be invariant on $\Gamma_N$ regardless of the choice of $\lambda$, which is the case if $\nu$ is not a multiple of vectors in $\Z^n$: we call such vectors {\it irrational}. For irrational $\nu$ we show that there is a unique linear profile the solutions converges to, by proving an oscillation lemma as well as using a quantitative version of Weyl's equi-distribution theorem.

 \medskip
 
 Secondly, to establish sufficient stability to address the general domains, some estimate on the convergence mode of $\e$- solutions in the strip domain, in terms of the variation of $\nu$, is necessary. Establishing this estimate is our second, and most challenging, main step (Theorem~\ref{continuity}).   For problems with divergence-form structure, such estimates were obtained in \cite{BLP} and \cite{GM} (also see \cite{LS}) by means of integral formulas. In our setting we must proceed by maximum principle-type arguments, which requires careful perturbation of the boundary data as well as construction of delicate barriers which describes the averaging behavior of solutions  up to the  Neumann boundary. The proof of Theorem 4.1 is based on the observation that, in ``meso-scopic" scale, the distribution of $g$ on two hyperplanes with close irrational normals are similar (see section 4.1. for a heuristic discussion of this fact). We adopt a multi-scale homogenization argument to address separately the effects on the solution caused by (a) microscopic oscillation of $g$ near the Neumann boundary (b) the difference in normal directions and (c) the oscillation present in the operator $F$.

\medskip

We mention that a significant difficulty arises due to the presence of the $\frac{x}{\e}$-dependence in $F$ in addition to the oscillations in $g$. To get around this difficulty, we use localization arguments as well as the existing homogenization results (e.g. \cite{CSW} and \cite{CS}) to show homogenization occurs away from the Neumann boundary. Such strategy works since near the Neumann boundary the oscillation of the first derivative ($g$) dominates the behavior of solutions. 

\vspace{15pt}

$\circ$ {\it Outline of the paper} 
\medskip

In section 2 we introduce some notations as well as preliminary results which will be used in the rest of the paper.
In section 3  we first study the averaging properties of the operator in the strip (half-space type) domains, to show that if the hypersurface is normal to an irrational direction then there is a unique homogenized Neumann data in the limit $\e\to 0$.  Besides the complications that the inhomogeities in the operator $F$ cause, the proof of averaging phenomena in this setting is due to the Weyl's distribution theorem, whose quantitative version that we need is borrowed from \cite{CKL} and is stated in Theorem~\ref{lemma-M}.
In section 4 we prove the main estimate on the mode of convergence of $\e$-solutions in the strip domain as the normal $\nu$ varies around a reference direction.  After a heuristic description of the proof, we prove a series of lemmas which  describe the behavior of solutions in different regions  of the strip domain, divided in terms of the distance to the Neumann boundary. Based on the lemmas then we are able to construct suitable barriers to obtain the desired estimate (Theorem 4.1).
Lastly in section 5 we prove Theorem~\ref{main} using the aforementioned estimate, the non-flat geometry of $\partial\Omega$, and the stability of viscosity solutions.

\section{Preliminaries}

\subsection{Rational and Irrational directions}
We first introduce the following categories of normal directions, following \cite{CKL}.
\begin{definition}
\begin{itemize}
\item[(i)] $\nu\in\mathcal{S}^{n-1}$ is a rational direction if $\nu\in \R\ZZ^n$.\\
\item[(ii)] $\nu\in\mathcal{S}^{n-1}$ is an irrational direction if $\nu$ is not a rational direction.
\end{itemize}
\end{definition}

Hyperplanes with irrational directions represent surfaces on which the boundary data does not change too much when one translates it in the normal direction. The following lemma quantifies this fact. The proof is a straightforward application of Weyl's equi-distribution theorem.

\begin{lemma} [Lemma 2.7, \cite{CKL}] \label{lemma-M}
For $\nu\in \RR^n$ and $x_0
\in \R^n$, let us define $H(x_0):=\{x: (x-x_0)\cdot\nu=0\}$.  Then the following is true for $0<\e<1$:

\begin{itemize}
\item[(i)] Suppose that $\nu$ is a rational direction. Then there exists a constant $M_\nu>0$ depending only on $\nu$ such that,  for any
$x \in H(x_0)$, there is $y \in H(x_0)$ satisfying
$$
|x-y| \leq M_\nu ; \quad y-x_0 \in  \ZZ^n.
$$

\item[(ii)] Suppose that $\nu$ is an irrational direction. Then there exists a mode of continuity $w_{\nu}:[0,1)\to\RR^+$ and a dimensional constant $M>0$ such that the
following is true: for any $x \in H(x_0)$, there exists $y \in \RR^n$
such that
$$|x-y| \leq M
 \e^{-9/10};\quad y-x_0 \in  \ZZ^n$$
and
\begin{equation}\label{continuity_mode}
 {\rm dist}(y, H(x_0)) <
\omega_{\nu}(\e).
\end{equation}
\item[(iii)] If $\nu$ is an irrational direction, then for any $x \in \RR^n$ and
$\delta>0$, there exists $y \in H(x_0)$  such that
$$|x-y| \leq \delta \hbox{ mod }\ZZ^n.
 $$
\end{itemize}
\end{lemma}
\begin{remark}
The power $-9/10$ appearing in (ii) of above lemma is arbitrarily chosen, and can be replaced with any $0<\alpha<1$. The mode of continuity $\omega_{\nu}$, of course, depends on the choice of $\alpha$.
\end{remark}

\subsection{Well-posedness and regularity results} 
Next we introduce the properties of the operator $F$. As before, let $\mathcal{M}^{n}$ denote the normed space of symmetric $n\times n$ matrices. In this paper we assume that the function $F(M,y):\mathcal{M}^{n}\times \R^n\to \R$ satisfies the following conditions:
\begin{itemize}
\item[(F1)][Homogeneity] $F(tM,x)=tF(M,x)$ for any $M\in\mathcal{M}^n$, $x\in\R^n$ and $t>0$. In particular it follows that $F(0,x)=0$.
\item[(F2)] [Lipschitz continuity] $F$ is locally Lipschitz continuous in $\mathcal{M}^n\times \R^n$, and there exists a constant $C>0$ such that for all $x,y\in\R^n$ and $M,N\in\mathcal{M}^n$
$$
|F(M,x)-F(M,y)| \leq C(|x-y|(1+\|M\|+\|N\|) + \|M-N\|).
$$
\item[(F3)][Uniform Ellipticity] There exists constants $0<\lambda<\Lambda$ such that 
\begin{equation}\label{elliptic}
\lambda\|N\| \leq F(M,x)-F(M+N,x) \leq \Lambda\|N\| \hbox{ for any }M,N\in\mathcal{M}^n \hbox{ and } N\geq 0.
\end{equation}
\item[(F4)] [Periodicity] For any $M\in\mathcal{M}^n$ and $x\in\R^n$, 
$$
F(M,x) = F(M, x+z) \hbox{ if } z\in\Z^n.
$$
\end{itemize}

\medskip

A typical example of $F$ satisfying (F1)-(F4) is the non-divergence type operator
\begin{equation}\label{div_elliptic}
F(D^2u, y) = -\Sigma_{i,j} a_{ij}(y) \partial_{x_ix_j} u ,
\end{equation}
where $a_{ij}:\R^n\to \R$ is  periodic, Lipschitz continuous and $A=(a_{ij})$ satisfies $\lambda(Id)_{n\times n}\leq A \leq \Lambda (Id)_{n\times n}$. Another example is the Bellman-Issacs operators arising from stochastic optimal control and differential games
$$
F(D^2u, y) =\inf_{\alpha\in A} \sup_{\beta\in B} \{\mathcal{L}^{\alpha\beta} u\},
$$
where $\mathcal{L}^{\alpha\beta}$ is a two-parameter family of uniformly elliptic operators of the form \eqref{div_elliptic}. We refer to \cite{cil} for further examples of $F$.

\medskip

Let $\Omega$ and $K$ be as given before, and let $f(y,\nu): \R^n\times S^{n-1} \to \R$ be a continuous function. To incorporate both $(P)_\e$ and the homogenized problem $(\bar{P})$, let us introduce a definition of viscosity solutions for the following problem:
$$
\left\{\begin{array}{lll}
F(D^2u,y) = 0 &\hbox{ in }& \Omega-K;\\ \\ 
\nu\cdot Du = f(x,\nu) &\hbox{ on } &\partial\Omega; \\ \\
u=1 &\hbox{ on } & K,
\end{array}\right.\leqno(P)_f
$$
where $\nu=\nu_x$ is the outward normal at the boundary point $x\in\partial\Omega$. See \cite{D} for a game-interpretation of $(P)_f$. 
The following definition is equivalent to the ones given in \cite{cil}:
\begin{definition}
\begin{itemize}
\item[(a)]  An upper semi-continuous function $u:\bar{\Omega}\to \R$ is a {\rm viscosity subsolution} of $(P)_f$ if $u$ cannot cross from below any $\phi\in C^2(\Omega)\cap C^1(\bar{\Omega})$ which satisfies
$$
F(D^2\phi,x) > 0 \hbox { in } \Omega-K; \quad\nu\cdot D\phi > f(x, \nu) \hbox{ on } \partial\Omega; \quad \phi >1\hbox{ on } K.
$$
\item[(b)] A lower semi-continuous function $u:\bar{\Omega}\to \R$ is a {\rm viscosity supersolution} of $(P)_f$ if  $u$ cannot cross from above any $\varphi\in C^2(\Omega)\cap C^1(\bar{\Omega})$ which satisfies 
$$
F(D^2\varphi,x) < 0 \hbox { in } \Omega-K; \quad\nu\cdot D\varphi  < f(x, \nu) \hbox{ on } \partial\Omega; \quad \varphi <1\hbox{ on } K.
$$
\item[(c)] $u$ is a {\rm viscosity solution} of $(P)_f$ if $u$ is both a viscosity sub- and supersolution of $(P)_f$.
\end{itemize}
\end{definition}

The existence and uniqueness of viscosity solutions in bounded domains are consequences of the following comparison principle.

\begin{theorem}\label{thm:comp} [Section V, \cite{il}]
Let $\Sigma$ be a domain in $\R^n$ whose boundary consists of $\Gamma_N$ and $\Gamma_D$, and let $F$ satisfy $(F1)-(F4)$.  Suppose $u, v:\bar{\Sigma}\to \R$ are bounded and continuous. If $u$ and $v$ satisfy
\begin{itemize}
\item[(a)] $F(D^2u, x) \leq 0 \leq F(D^2v, x)$ in $\Omega$;
\item[(b)] $u\leq v$  on $\Gamma_D$;
\item[(c)] $\frac{\partial u}{\partial\nu} \leq f(x) \leq \frac{\partial v}{\partial\nu}$  on $\Gamma_N$.
\end{itemize}
where $f(x):\RR^n\to \RR$ is continuous, and (a) and (c) should be interpreted in the viscosity sense. Then $u\leq v$ in $\Omega$.
\end{theorem}

Next we state regularity results for solutions of $F(D^2u, x)=0$ which will be used later. Let us recall the Pucci extremal operators: for a symmetric $n\times n$ matrix $M$, we define
$$
\mathcal{P}^+(M) := \Lambda(\Sigma_{e_i>0} e_i)+ \lambda(\Sigma_{e_i <0} e_i)
$$
and
$$
\mathcal{P}^-(M):=\lambda (\Sigma_{e_i>0} e_i )+ \Lambda(\Sigma_{e_i<0} e_i ),$$
where $\{e_i\}$'s are the eigenvalues of $M$. The operator $-\mathcal{P}^{\pm}$ then satisfies the assumptions (F1)-(F4). Moreover note that, due to \eqref{elliptic},  we have 
\begin{equation}\label{operator}
 -\mathcal{P}^+(M-N) \leq F(M,x) - F(N, x) \leq -\mathcal{P}^-(M-N) \hbox{ for any } M, N\in\mathcal{M}^n.
\end{equation}

In particular, the difference $w$ of two viscosity solutions of $F(D^2u,x)=0$  satisfies both $-\mathcal{P}^+(D^2 w) \leq 0$ and $-\mathcal{P}^- (D^2 w) \geq 0$ in the viscosity sense.

\begin{theorem}[Proposition 4.10, \cite{CC}: modified for our setting]\label{lemma-reg}
   Let $u$ be a viscosity solution of

   \begin{equation}\label{extreme}
-\mathcal{P}^+(D^2 u) \leq 0, \quad -\mathcal{P}^-(D^2 u) \geq 0
\end{equation}

in a domain $\Omega$. Then for any  $0<\alpha<1$ and a compact
subset  $\Omega'$ of $\Omega$, we have
$$
\|u\|_{C^{\alpha}(\Omega')}\leq Cd^{-\alpha}\|u\|_{L^{\infty}(\Omega)}<\infty
$$

with $d=d(\Omega', \partial\Omega)$ and $C>0$ depending only on $n,\lambda,\Lambda$.
 \end{theorem}
\vspace{10pt}

\begin{theorem}[Theorem 8.3, \cite{CC}] \label{lemma-interior}
Let $F$ satisfy $(F1)-(F4)$, and $u$ be a continuous function in $\bar{B}_1$ solving
$$
F(D^2 u,x)=0 \hbox{ in } B_1(0).
$$
 Then there exists $0<\alpha<1$ and $C>0$, depending only on $\Lambda, \lambda$ and $n$, such that
 $$
 \|u\|_{C^{1,\alpha}(B_{1/2}(0))}\leq C\|u\|_{L^\infty(B_1)}.
 $$
\end{theorem}

\vspace{10pt}

\begin{theorem}[Theorem 8.2, \cite{MS}: modified for our setting]\label{thm:reg2}
Let
$$
B_r^+:= \{|x|< 1\} \cap\{x\cdot e_n \geq 0\}\hbox{ and }\Gamma = \{|x|<1\}\cap\{x\cdot e_n = 0\}.
$$
For given bounded function $g: B_1^+\to \R$,  the following is true for $u$ satisfying
$$
\left\{
\begin{array}{lll}
-\mathcal{P}^+(D^2 u) \leq 0, \quad -\mathcal{P}^-(D^2 u) \geq 0 &\hbox{ in }& B_1^+;\\
\nu\cdot Du = g \hbox{ on } \Gamma.
\end{array}\right.
$$ 

\begin{itemize}
\item[(a)] $u\in C^{\alpha}(\overline{B^+_{1/2}})$ for $ 0<\alpha=\alpha(n,\lambda,\Lambda)<1$ given in Theorem~\ref{lemma-interior} with the estimate
$$
\| u\|_{C^{\alpha}(\overline{B^+_{1/2}})} \leq C(\|u\|_{L^\infty(\overline{B_1^+})}+ {\rm max}\|g\|).
$$

\item[(b)] If, in addition, $ g\in C^{\beta}(B_1^+)$ for some $0<\beta<1$ and $u$ satisfies $F(D^2u,x)=0$, then $u\in C^{1,\alpha} (\overline{B^+_{1/2}}),$
where $\alpha = \min(\alpha_0, \beta)$ with $\alpha_0$ given in (a). Moreover, we have the estimate
$$
\| u\|_{C^{1,\alpha}(\overline{B^+_{1/2}})} \leq C(\|u\|_{L^\infty(\overline{B_1^+})}+ \|g\|_{C^\beta(\Gamma)}),
$$
where $C$ is a constant depending only on $n, \lambda, \Lambda$ and $\alpha$.
\end{itemize}
\end{theorem}

\subsection{Localization}
Here we discuss how to localize solutions of Neumann-boundary problems. For given $\nu\in\mathcal{S}^{n-1}$, let us define
$$
\Pi_{\nu}(0):=\{ x:-1\leq  x\cdot\nu \leq 0\}, \quad\Gamma_N = \{x\cdot\nu = 0\}, \quad\Gamma_D = \{x\cdot\nu = -1\}.
$$

\begin{lemma}\label{localization}
For given $R,\e>0$, suppose $u:\overline{\Pi_{\nu}(0)}\to \R$ solves the following:
$$
\left\{\begin{array}{ll}
-\mathcal{P}^+(D^2 u) \leq 0, \quad -\mathcal{P}^-(D^2 u)\geq 0 &\hbox{ in } \Pi_\nu(0)\cap\{|x|\leq R\},\\ \\
\partial_{\nu} u = g(x) &\hbox{ on } \Gamma_N\cap\{|x|\leq R\};\\ \\
u=f(x) &\hbox{ on } \Gamma_D\cap\{|x|\leq R\},\\ \\
|u|(x) \leq R^{2-\e} &\hbox{ on } |x|=R.
\end{array}\right.
$$
 If  $|f(x)|,|g(x)|\leq \delta$, then we have
$$
|u(x)| \leq 2\delta+16(n-1)\frac{\Lambda}{\lambda}R^{-\e} \hbox{ in } \Pi_{\nu}(0) \cap\{|x|\leq 1\}.
$$
\end{lemma}
\begin{proof}
 Without loss of generality, let us  set $\nu=e_n$, $x' = (x_1,...,x_{n-1})$, .and consider the function
$$
\varphi(x) =(\delta+4(n-1)\frac{\Lambda}{\lambda}R^{-\e}) (x_n+2) +\frac{2}{R^{\e}}{\Large(}|x'|^2-(n-1)\frac{\Lambda}{\lambda} ((x_n+1)^2-1){\Large)}.
$$
 Then  $-\mathcal{P}^+(D^2\varphi) =0$, and $\varphi\geq \delta $ on $\{x_n=-1\}$, $\varphi \geq R^{2-\e}$ on $|x|=R$, and $\partial_{e_n} \varphi \geq \delta $ on $\{x_n=0\}$.  Therefore by comparison principle (Theorem~\ref{thm:comp}) we have 
$$
u \leq \varphi \hbox{ in } \Pi_\nu(0) \cap\{|x|\leq R\}.
$$
Similarly one can show that $u \geq -\varphi$ in $\Pi_\nu(0)\cap \{|x|\leq R\}$.  Now  we conclude by evaluating $|\varphi|$ in the region $\Pi_{\nu}(0) \cap\{|x|\leq 1\}$.

\end{proof}

The following holds as a corollary of above lemma with $\delta=0$:

\begin{corollary}\label{localization2}
Let  $\omega$ satisfy the following in the viscosity sense:
\begin{itemize}
\item[(a)] $-\mathcal{P}^+(D^2\omega) \leq 0$, $-\mathcal{P}^-(D^2\omega) \geq 0$ in $\Sigma_R:= \Pi_\nu(0) \cap \{|x'| \leq R\}$.
\item[(b)] $\partial\omega/\partial\nu =0$ on $\Gamma_N \cap \Sigma_R$;
\item[(c)] $\omega=0$ on $\Gamma_D\cap\Sigma_R$;
\item[(d)] $|\omega| \leq R^{2-\e}$.
\end{itemize}

Then we have
$$
|\omega| \leq C R^{-\e} \hbox{ in } \Pi_{\nu}(0) \cap\{|x'| \leq 1\}.
$$
\end{corollary}

From Corollay~\ref{localization2}, the following holds.

\begin{theorem}[general comparison]\label{general:cp}
Suppose $u,v:\overline{\Pi_\nu(0)}\to \R$ satisfies the following in the viscosity sense:
\begin{itemize}
\item[(a)] $F(D^2u,x) \leq 0\leq F(D^2 v,x) $ in $\Pi_{\nu}(0)$;
\item[(b)] $u\leq v$ on $\Gamma_D$, \quad $|u|,|v| \leq |x|^{2-\e}$ \hbox{ for large} $|x|$;
\item[(c)] $\frac{\partial u}{\partial\nu} \leq f(x) \leq \frac{\partial v}{\partial\nu}$ on $ \Gamma_N$,
\end{itemize}
where $f(x):\RR^n\to\RR$ is continuous. Then $u\leq v$ in $\Pi_{\nu}(0)$.
\end{theorem}

\subsection{Homogenization of the operator}

Here we state the homogenization result obtained in \cite{CS}.

\begin{theorem}[\cite{CS}]\label{original_convergence}
Let $F$ satisfy $(F1)-(F4)$. Then there exists a unique operator $\bar{F}(D^2u):\mathcal{M}^n\to \R^n$ satisfying $(F1)-(F3)$ such that the following is true:

\medskip

Let $\Sigma$ be a bounded domain in $\R^n$ whose boundary is $C^1$ and consists of two nonempty parts $\Gamma_D$ and $\Gamma_N$, and suppose that $g,f:\R^n\to \R$ are continuous. If $u^\e$ solves
$$
\left\{\begin{array}{ll}
F(D^2u^\e, \frac{x}{\e})=0 &\hbox{ in } \Sigma;\\ \\
\partial_{\nu} u^\e = g(x) &\hbox{ on } \Gamma_N;\\ \\
u^\e=f(x) &\hbox{ on } \Gamma_D
\end{array}\right.
$$
Then $u^\e$ uniformly converges to the solution $u$ of
$$
\left\{\begin{array}{ll}
\bar{F}(D^2u)=0 &\hbox{ in } \Sigma,\\ \\
\partial_{\nu} u = g(x) &\hbox{ on } \Gamma_N;\\ \\
u=f(x) &\hbox{ on } \Gamma_D
\end{array}\right.
$$
Moreover there exists $0<\alpha<1$ such that
\begin{equation}\label{estimate100}
|u^\e-u| \leq C\e^{\alpha}.
\end{equation}
\end{theorem}

\begin{remark}
By Lemma~\ref{general:cp}, the theorem is valid in the setting of the unbounded domain $\Omega=\Pi_{\nu}(0)$ with $\Gamma_D=\{x\cdot\nu=-1\}$, $\Gamma_N=\{x\cdot\nu=0\}$ and with uniformly bounded $u^\e$, though in this setting we would not have the quantitative estimate \eqref{estimate100}.

\end{remark}

Using above result we can apply compactness arguments to show the following.

\begin{theorem}\label{ext}
Let $K$ be a positive constant and let $f:\R^n\to\R$ be bounded and H\"{o}lder continuous.  For given $\nu\in\mathcal{S}^{n-1}$, let $u_N:\{-K\leq x\cdot\nu\leq 0\}\to \R$ be the unique bounded viscosity solution of  
$$
\left\{\begin{array}{lll}
F(D^2 u_N, Nx)=0 &\hbox{ in }& \{-K\leq x\cdot\nu \leq 0\} \\ \\
\partial_{\nu} u_N = f(x) &\hbox{ on }& \{x\cdot\nu=0\},\\ \\
u = 1 &\hbox{ on } & \{x\cdot\nu=-K\} 
\end{array}\right.\leqno(P_N).
$$
Then for any $\delta>0$, there exists $N_0$ only depending on $K$, the bound and H\"{o}lder exponent of $f$ such that
\begin{equation}\label{conv}
|u_N - \bar{u}|(x) \leq \delta \hbox{ in } \{|x|\leq K\} \quad\hbox{ for } N \geq N_0,
\end{equation}
where $\bar{u}$ is the unique bounded viscosity solution of 
$$
\left\{\begin{array}{lll}
\bar{F}(D^2 \bar{u})=0 &\hbox{ in }& \{-K\leq x\cdot\nu \leq 0\};\\ \\
\partial_{\nu} \bar{u} = f(x) &\hbox{ on }& \{x\cdot\nu=0\};\\ \\
u=1 &\hbox{ on } & \{x\cdot\nu = -K\}.
\end{array}\right.
$$
\end{theorem}

\begin{proof}
1. For fixed $\nu$, this is a consequence of Theorem~\ref{original_convergence} and Lemma~\ref{localization}.

\vspace{10pt}

2. Let us fix $K$ and $\delta$. To show that $N_0$ is independent of $\nu$,  suppose not. Then for some $\delta>0$ there exists a sequence of directions $\nu_k$ and a sequence of  uniformly bounded, uniformly H\"{o}lder continuous functions $f_k \in C^{\beta}(\R^n)$ such that $\nu_k \to \nu_0$, $f_k(x)$ locally uniformly converges to $\bar{f}(x)$ and the the constants $N_k=N_0(\nu_k)$ given by \eqref{conv} goes to infinity as $k\to\infty$.  Let us denote $u_N^k$ and $\bar{u}^k$ by the corresponding solutions given in above theorem with $\nu=\nu_k$. Since $\{u_{N_k}^k\}$ is uniformly H\"{o}lder, along a subsequence it converges to a function $u_0$. Using the stability properties of viscosity solutions as well as Theorem~\ref{original_convergence}, one can check that $u_0$ solves 
$$
\left\{\begin{array}{lll}
\bar{F}(D^2 \bar{u})=0 &\hbox{ in }& \{-K\leq x\cdot\nu_0 \leq 0\};\\ \\
\partial_{\nu} \bar{u} = \bar{f}(x) &\hbox{ on }& \{x\cdot\nu_0=0\};\\ \\
u=1 &\hbox{ on } & \{x\cdot\nu = -K\}.
\end{array}\right.
$$




Now we choose $N_0$ so that \eqref{conv} holds for $\nu=\nu_0$ and $\delta/3$, yielding a contradiction.

\end{proof}

\section{Homogenization in half-spaces}
Now we are ready to study the averaging properties for solutions of $(P_\e$), beginning with the strip domain: later we will approximate the general domain with these solutions.  For given $\nu\in\mathcal{S}^n$, let us define 
 \begin{equation}\label{strip}
 \Pi_{\nu}(p):=\{-1\leq (x-p)\cdot\nu \leq 0\}.
 \end{equation}
 Consider $u^\e$ solving
$$
\left\{\begin{array}{lll}
F(D^2u^\e,\frac{x}{\e})=0 &\hbox{ in } & \Pi_{\nu}(p);\\ \\
\frac{\partial}{\partial \nu} u = g(\frac{x}{\e}) &\hbox{ on }& \Gamma_0(\nu,p):=\{(x-p)\cdot \nu =0\};\\ \\
u=1 &\hbox{ on }& \Gamma_1(\nu,p):=\{(x-p)\cdot\nu=-1\}.
\end{array}\right.\leqno(P_\e^{\nu})
$$

The main theorem in this section is given below:

\begin{theorem} \label{thm:planar}
 For  given $\nu\in\mathcal{S}^n$ and $p\in \RR^n$, let $u_\e$ solve $(P^\nu_\e)$. Then the following holds:
\begin{itemize}
\item[(i)] For irrational directions $\nu$, there exists a unique constant $\mu(\nu)\in [\min g, \max g]$ independent on the choice of $p$
 such that $u^\e$ locally uniformly converges to the linear profile
 \begin{equation}\label{linear2}
 u(x)= \mu(\nu)((x-p)\cdot\nu+1)+1.
\end{equation}

\item[(ii)] For rational directions $\nu$, if  $\Gamma_0$ goes through the origin (that is if $p =0$),
then the statement in (i) holds for $\nu$ as well.
\item[(iii)] [Error estimate] Let $\nu$ be an irrational direction.
Then for $u^\e$ solving $(P^{\nu}_\e)$ and $u$ as given in \eqref{linear2}, we have
the following estimate: for any $0<\alpha<1$, there exists a
constant $C=C_{\alpha}>0$ such that
\begin{equation}\label{error101}
|u^\e - u| \leq  C\e^{\alpha/20} + C\omega_\nu(\e)^{\beta}\quad \hbox{  in
} \Pi_\nu(p),
\end{equation}
where $\omega_\nu$ is as given in Lemma~\ref{lemma-M} (ii).
\end{itemize}
\end{theorem}

The proof of Theorem~\ref{thm:planar} consists of several lemmas below: the outline largely follows that of the corresponding result (Theorem 2.4) in \cite{CKL}: for clarity we present the full proof below.

\vspace{0.2 in}
 {\bf  Proof of Theorem~\ref{thm:planar}.}
\vspace{0.1in}

Due to the uniform H\"{o}lder regularity of
$\{u_\e\}$ (Theorem~\ref{thm:reg2}(a)), along subsequences
 $u_{\e_j} \to u$  in  $\overline{\Pi_\nu (p)}$.
  Note
that there may be different limits along different subsequences
${\e_j}$. Below we will show that if $\nu$ is an irrational
direction then all subsequential limits of $\{u_\e\}$ coincide.

\medskip

First we present a ``flatness" result: we show that,  sufficiently ($\e^{1/20}$-) away from the Neumann boundary $\Gamma_0=\Gamma(\nu,p)$, $u_\e$ is almost
 a constant on hyperplanes parallel to $\Gamma_0$.

\begin{lemma} \label{Claim 2}
Let us fix $\nu\in\mathcal{S}^n - \R\Z$ and $p\in\R^n$, and let us denote $\Gamma_0=\Gamma_0(\nu,p)$.
 Then for any  $x_0 \in \Pi_{\nu}(p)$ with  ${\rm dist}(x_0, \Gamma_0)>\e^{1/20}$, and for any $0<\alpha<1$, 
 there exists a constant $C=C(\alpha,n)$ such
that for any $x \in H(x_0):=\{(x-x_0)\cdot\nu=0\}$
\begin{equation}\label{mid-flat}
|u_\e (x)-u_\e(x_0)| \leq  \mathcal{E}(\e),
\end{equation}
where $\mathcal{E}(\e):= C \e^{\alpha/20} + C \omega_\nu(\e)^\beta$
with $\omega_\nu:[0,\infty)\to[0,\infty)$ as
given in  Lemma~\ref{lemma-M} (ii).
\end{lemma}

\begin{proof}

Let $x_1\in H(x_0)$. By Lemma~\ref{lemma-M} (ii), there exists $y\in\R^n$ such that $|x_1-y| \leq M\e^{1/10},$  $y-x_0\in\e\Z^n$ and 
\begin{equation}\label{distance}
dist(y, H(x_0)) < \e \omega_{\nu}(\e).
\end{equation}

Let us compare $u_\e(x)$ with $\tilde{u}_\e(x):= u_\e(x-x_0+y)$ in  the domain $\Sigma = \Pi_\nu(p) \cap \Pi_\nu(p+x_0-y)$. Due to Theorem~\ref{thm:reg2}, \eqref{distance} and the fact that $y-x_0\in\e \Z^n$ it follows that 
\begin{equation}\label{estimate1}
 |\partial_\nu u_\e(x) - g(\frac{x-x_0+y}{\e}) |=|\partial_\nu u_\e(x) - g(\frac{x}{\e})| \leq \omega_{\nu}(\e)^{\alpha} \hbox { on } \Gamma_0 - |(x_0-y)\cdot\nu|\nu.
 \end{equation}
 Moreover, due to the H\"{o}lder continuity of $u_\e$ up to the Dirichlet boundary $\Gamma_1(\nu,p)$ we have 
 $$
 |u_\e(x)-1 | \leq (\e\omega_{\nu}(\e))^\beta \hbox{ on } \Gamma_1(\nu,p) + |(x_0-y)\cdot\nu|\nu
 $$
  Lastly, due to the fact that $1\leq g\leq 2$, we have $|u_\e(x)| \leq 2$ in $\Pi_\nu(p)$.
  
  \medskip
  
Putting together above estimates, we can apply Lemma~\ref{localization}  to $v(x):= u_\e(x) - \tilde{u}_\e(x)$ with $\delta = \omega_{\nu}(\e)^\alpha$ and $R >> \delta$ to  obtain that 
 \begin{equation}\label{estimate2}
 |u_\e(x) - \tilde{u}_\e(x)| \leq  C\omega_{\nu}(\e)^{\alpha} \hbox { in } |x-x_0| \leq 1.
 \end{equation} 
Now we can conclude, by observing
$$
\begin{array}{lll}
|u_\e(x_0)-u_\e(x_1) | & \leq&  |u_\e(x_0) - u_\e(y)| + |u_\e (y)-u_\e(x_1)|\\ 
                                    & \leq & C\omega_{\nu}(\e)^\alpha  +|u_\e(y)-u_\e(x_1)|\\
                                    & \leq & C\omega_{\nu}(\e)^\alpha   + C\e^{-\alpha/20} (M\e^{1/10})^\alpha\\
                                    & \leq & C\omega_{\nu}(\e)^\alpha + C\e^{\alpha/20},
\end{array}
$$
where the second inequality is due to \eqref{estimate2} with $x=x_0$ and the third inequality is due to the fact that $|x_1-y| \leq M\e^{1/10}$ and Theorem~\ref{lemma-reg}.

\end{proof}

\medskip

Next we define the average linear profile for $u^\e$: pick a point $y_0\in \{x\cdot\nu = -\frac{1}{2}\}$ and  for each $\e>0$ let us define $v_\e$ by 
$\mu_j$  such that 
\begin{equation}\label{linear}
v_{\e}(x) := \mu_\e x\cdot\nu +u_\e(y_0), \quad \mu_\e = \mu(u^\e) : = u_\e(y_0)-1.
\end{equation}

\vspace{10pt}

By Lemma~\ref{Claim 2} and by the comparison principle
(Theorem~\ref{thm:comp}), we obtain that
\begin{equation}\label{comp1}
\underline{U} \leq u_{\e} \leq \bar{U}
\hbox{ in } \{-1\leq (x-p)\cdot\nu \leq -\e^{1/20}\},
\end{equation}
 where $\bar{U}$ and $\underline{U}$ are linear functions which satisfies
$\bar{U}=\underline{U}=1$ on $\{(x-p)\cdot\nu=-1\}$ and 
$$
\bar{U} = u(x_0)+\mathcal{E}(\e), \quad \underline{U}=u(x_0)-\mathcal{E}(\e) \hbox{ on } \{(x-p)\cdot\nu=-\e^{-1/20}.\}
$$

From \eqref{comp1} and the definition of $v_j$ the following estimate holds: For  $x
\in \Pi_\nu(p)$,
\begin{equation} \label{C2}
|u_{\e} (x)-v_\e(x)| \leq  \mathcal{E}(\e) + C\e^{1/20}.
\end{equation}
this implies that $\{u_\e\}$ uniformly converges to linear profiles along subsequences.

\vspace{0.1in}

\begin{lemma}\label{Claim 4}[Theorem 3.1, (i)]
 If $\nu$ is an irrational direction, then there exists a unique $\mu_\nu$ which only depends on $\nu$ such that  $u_\e$ uniformly converges to a linear profile $u(x) = \mu_\nu ((x-p)\cdot\nu+1)+1$.
\end{lemma}

\begin{proof}
 1. Let $0<\eta < \e$ be sufficiently small, let $p_1,p_2\in \R^n$, and let $u_\e=u_{\e,p_1}$ solve $(P_\e^{\nu})$ in $\Pi_\nu(p_1)$ and $u_{\eta}=u_{\eta,p_2}$ solve $(P_\e^{\nu})$ in $\Pi_\nu(q)$. We will show that $|\mu_\e - \mu_{\eta}|$ can be made arbitrarily small if we choose $\e$ sufficiently small, independent of the choice of $p_1$ and $p_2$. This concludes Theorem 3.1 (i). In the proof below we strongly use the fact that $F(0,x)\equiv 0$ and thus that linear functions solve $F(D^2u,x)=0$. 
 
 \medskip
 
 2. Let  us  re-scale 
 $$
 w_{\e}(x) = \displaystyle{\frac{u_\e(\e
x)}{\e}}, \,\,\,\,\,  w_\eta(x)=\displaystyle{ \frac{u_\eta(\eta
x)}{\eta}} $$ and denote by $\Gamma_1 = \frac{1}{\e}\Gamma_0(\nu,p_1)$ and $\Gamma_2=\frac{1}{\eta}\Gamma_0(\nu, p_2)$ as the
corresponding Neumann boundaries of $w_{\e}$ and $w_{\eta}$.  We first translate $\Gamma_1$ by $\tau\in\ZZ^n$ such that $\Gamma_1$ and $\Gamma_2$ are close. By (iii) of Lemma~\ref{lemma-M}, there exist $q_1 \in \Gamma_1$ such that
 $$\frac{p_2}{\eta}=q_1 +\tau+ y_1,
 $$ where 
$|y_1| \leq \eta$ and $\tau \in \ZZ^n$. Hence after translating $\omega_\e$ by $\tau$, we may suppose that
  $w_{\e}(x)$ and
$w_\eta(x)$ are  defined, respectively,  on the extended strips
$$
\Omega_\e:=\{x: -\frac{1}{\e} \leq (x-y_1)\cdot\nu \leq 0\}\quad\hbox{ and }\quad
\Omega_\eta:=\{x: -\frac{1}{\eta} \leq (x-y_2)\cdot\nu \leq 0\},
$$
where $y_2 = \frac{p_2}{\eta}$ and $|y_1-y_2 |\leq \eta$. 

\medskip

 Without loss of generality, we may assume that $y_2\cdot \nu \leq y_1 \cdot\nu$. Since $|y_1-y_2|\leq \eta$ and $g\in C^{\beta}$, we have 
\begin{equation}\label{no}
|g(x) - g(x+y_1-y_2)|\leq \eta^{\beta}.
\end{equation}
From \eqref{no} and the $C^{1,\alpha}$ regularity of $w_\e(x)$ (see Theorem~\ref{thm:reg2}), we conclude that there exists a constant $C_0>0$ independent of $\e$ and $\eta$ such that 
\begin{equation}\label{boundary2}
|\partial_\nu w_\e(x) - g(x) |\leq C_0(\eta^{\alpha}+\eta^{\beta})\quad \hbox{ on } \Gamma_\eta:=\{(x-y_2)\cdot\nu=0\}.
\end{equation}

 Let $v_\e$ be given by \eqref{linear}.  Then by \eqref{C2}
\begin{equation}\label{above0}
|w_\e(x)-\frac{v_\e(\e x)}{\e}| \leq \frac{\mathcal{E}(\e)}{\e}.
\end{equation}
 \vspace{10pt}

From \eqref{above0} and the comparison principle, it follows that
\begin{equation}\label{above}
(\mu_\e-\mathcal{E}(\e))((x-y_1)\cdot\nu  + \dfrac{1}{\e})\leq w_\e(x)
-\dfrac{1}{\e} \leq (\mu_\e +\mathcal{E}(\e))((x-y_1)\cdot\nu
+\dfrac{1}{\e}) \quad \hbox{ in } \Omega_\e.
\end{equation}  \eqref{above} means that the slope of $w_\e$ in the direction of
$\nu$ (i.e. $\nu\cdot Dw_\e$) is between $\mu_\e\pm\mathcal{E}(\e)$ on $\{x: (x-y_1)\cdot\nu= -\frac{1}{\e}\}$. Now
let us consider the linear profiles
$$
l_1(x) = a_1(x-y_1)\cdot\nu+b_1\hbox{ and
}l_2(x)=a_2(x-y_1)\cdot\nu+b_2,
$$ whose respective slopes are
$a_1=\mu_\e+\mathcal{E}(\e)$ and $a_2=\mu_\e - \mathcal{E}(\e)$.  Here $b_1$
and $b_2$ are chosen to match the Dirichlet boundary data of $\omega_{\eta}$.

\vspace{10pt}

3. Now we define
$$
\overline{w}(x): =  \left\{\begin{array} {lll} l_1(x) &\hbox{ in } &
\{-1/\eta-\eta \leq(x-y_1)\cdot \nu \leq -1/\e\}\\ \\
w_\e(x)+c_1  &\hbox{ in }& \{-1/\e \leq(x-y_1)\cdot \nu \leq 0\}
\end{array}\right.
$$
and
$$
\underline{w}(x): =  \left\{\begin{array} {lll} l_2(x) &\hbox{ in }
&
\{-1/\eta-\eta \leq(x-z_1)\cdot \nu \leq -1/\e\}\\ \\
w_\e(x)+c_2  &\hbox{ in }& \{-1/\e \leq(x-y_1)\cdot \nu \leq 0\}
\end{array}\right.
$$
where $c_1$ and $c_2$ are constants satisfying 
$$l_1=w_\e+c_1, \quad l_2=w_\e+c_2\hbox{ on }\{(x-y_1)\cdot\nu= -1/\e\}.
$$ (See Figure 2.)

\begin{figure}
\center{\epsfig{file=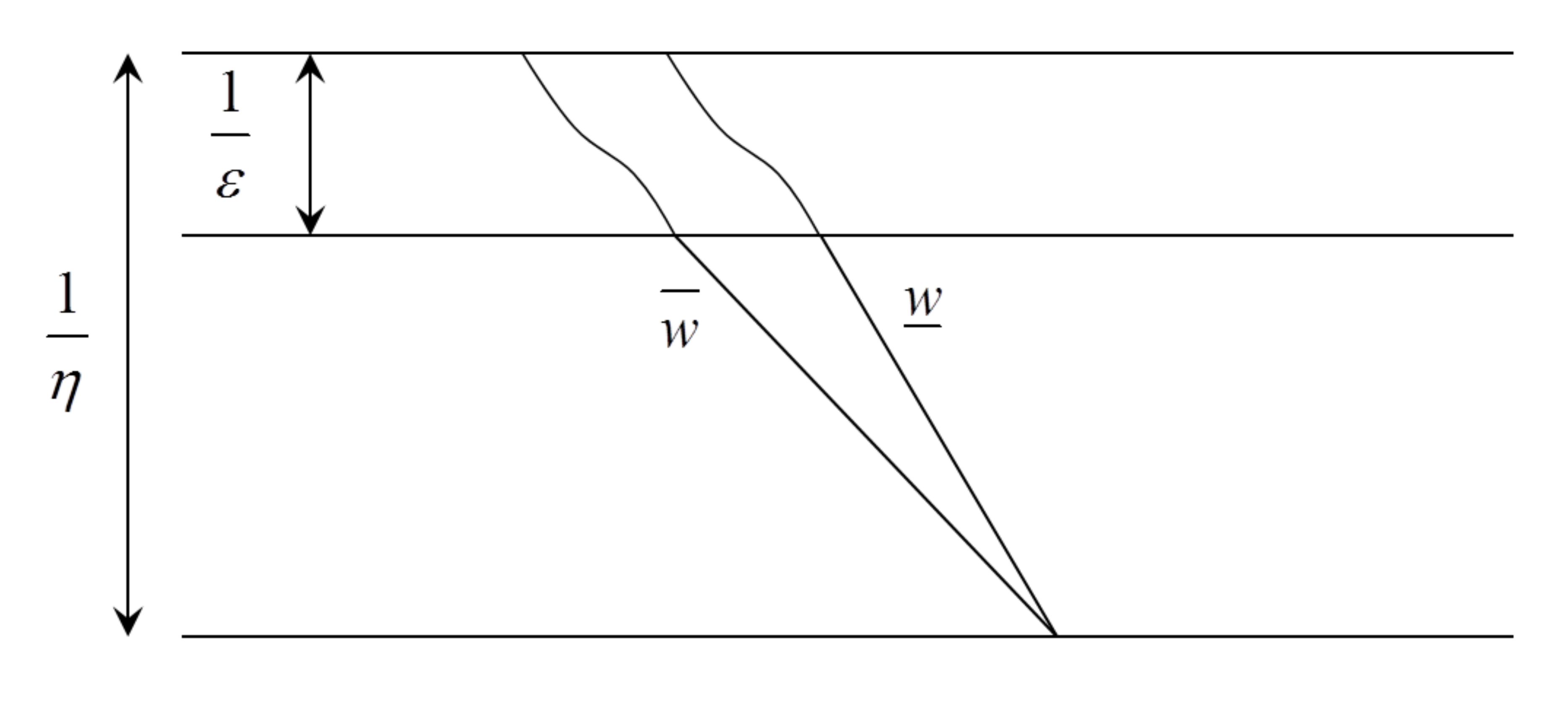,height=2.0in}} \center{Figure 2}
\end{figure}

\medskip

 Note that, due to \eqref{above},  in $\{-\frac{1}{\e}\leq (x-z_1)\cdot\nu \leq 0\}$ we have
 $$
 \overline{w}(x)= \min (l_1(x), w_\e(x)+c_1)\quad\hbox{ and }\quad\underline{w}(x) = \max (l_2(x), w_\e(x)+c_2),
 $$
 and thus it follows that $\overline{w}$ and $\underline{w}$ satisfies, in the viscosity sense,
 $$
F(D^2\underline{w}, \frac{x}{\e}) \leq 0\leq  F(D^2\bar{w},\frac{x}{\e}) \quad \hbox{ in } \Omega_{\eta}.
$$

\medskip

 4. Let us define
$$
h_1(x)=C_0(\eta^\alpha+\eta^\beta )((x-y_2)\cdot\nu + 1/\eta),
$$
where $C_0$ is given in \eqref{boundary2}.
Then  $w^+:=\overline{w}+h_1$ and $w^-:=\underline{w}-h_1$
respectively solves 
$$
\left\{\begin{array}{lll}
F(D^2w^+,x) \geq 0 &\hbox{ in }&\Omega_\eta;\\ \\
\dfrac{\partial w^+}{\partial \nu} \geq g(x) &\hbox{ on }&
\Gamma_\eta
\end{array}\right.
$$
and
$$
\left\{\begin{array}{lll}
F(D^2w^-,x) \leq 0 &\hbox{ in }&\Omega_\eta;\\ \\
\dfrac{\partial w^-}{\partial \nu} \leq g(x) &\hbox{ on }
&\Gamma_\eta.
\end{array}\right.
$$

 Since $w^+=w^- = w_{\eta}$ on $\{(x-p)\cdot\nu=-\frac{1}{\eta}\}$,
 from the comparison principle for $(P_\e)$ (Theorem~\ref{general:cp} it follows that
\begin{equation}\label{error}
 w^- \leq w_{\eta} \leq w^+\,\, \hbox{ in }
 \Omega_\eta.
\end{equation}

Hence we conclude
\begin{equation}\label{error2}
 |\mu_\eta -\mu_\e| \leq \mathcal{E}(\e) +C(\eta^\alpha+\eta^\beta),
\end{equation}
 where $\mu_\eta$ is the slope of $v_\eta$, which is defined as in \eqref{linear} for $u_\eta$. Now we can conclude.

\end{proof}

\vspace{10pt}

The proof of the following lemma is immediate from Lemma~\ref{Claim
4} and \eqref{error2} .

\medskip

\begin{lemma}\label{irrational_error} [Error estimate: Theorem~\ref{thm:planar} (iii)]
For any irrational direction $\nu$ there is a unique homogenized
slope $\mu(\nu)\in \RR$ and $\e_0=\e_0(\nu)>0$ such that for
$0<\e<\e_0$ the following holds: for any $p\in\R^n$ and $u_\e$ solving $(P_\e^\nu)$ in $\Pi_\nu(0)$, 
 \begin{equation}\label{error:final}
 |u_\e(x)-(\mu(\nu)((x-p)\cdot\nu +1)+1)| \leq \mathcal{E}(\e) \hbox{ in } \Pi_{\nu}(p).
 \end{equation}

\end{lemma}

\begin{lemma} \label{Claim 4'} [Theorem~\ref{thm:planar} (ii)]

Let $\nu$ be a rational direction. If the Neumann boundary
$\Gamma_0$ passes through $p=0$, then there is a unique homogenized
slope $\mu(\nu)$ for which the result of
Lemma~\ref{irrational_error} holds with $\mathcal{E}(\e)
=C\e^{\alpha/2}.$

\end{lemma}

\begin{proof} The proof is parallel to that of  Lemma~\ref{Claim 4}.  Let $\omega_\e$ and $\omega_{\eta}$ be as given in the proof of Lemma ~\ref{Claim 4}. Note that, since
$\Omega_\e$ and $\Omega_\eta$ have their Neumann boundaries passing
through the origin, $\partial w_\e /\partial \nu =g(x) =\partial
w_\eta /\partial \nu$ without translation of the $x$ variable, and
thus we do not need to use the properties of hyperplanes with an
irrational normal (Lemma~\ref{lemma-M}(iii)) to estimate the error between the
shifted Neumann boundary datas.

\end{proof}

As mentioned in \cite{CKL}, if $\nu$ is a rational direction with
$p\neq 0$, the values of $g(\cdot/\e)$ on $\partial \Omega_\e$ and
$\partial \Omega_\eta$ may be very different under any translation,
and thus the proof of Lemma~\ref{Claim 4} fails. In this case $u_\e$
may converge to solutions of different Neumann boundary data
depending on the subsequences.

\section{Continuity of the homogenized slope}

In the previous section we have shown that for an irrational direction $\nu \in \mathcal{S}^{n-1} -
\RR\ZZ^n$, there is a unique homogenized slope $\mu(\nu)$ for any solutions of $(P_\e^\nu)$ in $\Pi_\nu(p)$. In this section we investigate the continuity properties of $\mu(\nu)$ as well as the mode of convergence for $u^\e$ as the normal direction $\nu$ of the domain varies. For section 4 and 5 we assume the following additional condition on the homogenized operator $\bar{F}$ as given in Theorem ~\ref{original_convergence}:

\begin{equation}\label{new}
\bar{F}(M):\mathcal{M}^n\to \R \hbox{ only depends on the eigenvalues of $M$.}
\end{equation}

As mentioned before, the condition~\eqref{new} is equivalent to saying that $\bar{F}$ is rotation and reflection invariant.

\begin{theorem}\label{continuity}
Let $\mu(\nu): (\mathcal{S}^{n-1} - \RR\ZZ^n) \to \RR$ be as given in Theorem 3., and suppose that $\bar{F}$ satisfies \eqref{new}. Then $\mu$ has a continuous
extension $\bar{\mu}(\nu): \mathcal{S}^{n-1} \to \RR$. More
precisely for any $\nu\in S^{n-1}$ and $\delta>0$  there exists $\e_0=\e_0(\nu)$ such that the following holds:

If $\nu_1$ and  $\nu_2$ are irrational such that
\begin{equation}\label{small}
0< |\nu_1-\nu| ,|\nu_2-\nu| < \e_0,
\end{equation}
then we have
\begin{itemize}
\item[(a)]$|\mu(\nu_1)-\mu(\nu_2)|<\delta$;
\item[(b)] For given $p\in\R^n$, the solutions
$u_\e^{\nu_i}$ of $(P^{\nu_i}_\e)$ in $\Pi_{\nu_i}(p)$ and the average slope $\mu(u_\e^{\nu_i})$ given as in \eqref{linear} satisfies 
$$|\mu(u_\e^{\nu_i}) - \mu(\nu_i)|
< \delta \quad\hbox{ if }\,\,\e \leq \delta|\nu_i-\nu|^{21/20}, \quad i=1,2.
$$

\end{itemize}
\end{theorem}

\begin{remark}
In the proof we indeed show that, for  any directions $\nu_1$ and $\nu_2$ satisfying \eqref{small}, the range of $\{\mu(u_\e^{\nu_i})\}_{\e, i}$ fluctuate only by $\delta$, if $\e$ is sufficiently small (see \eqref{conclusion}). The fact that $\nu_i$'s are irrational is only used to guarantee that there is only one subsequential limit for $\mu(u_\e^{\nu_i})$, with $i=1,2$. Note that our statement uses $\nu$ as the reference direction and thus is a different type of estimate than those in Theorem~\ref{thm:planar}).
\end{remark}

\medskip

Thanks to Lemma~\ref{irrational_error}, it is enough to consider the case $p=0$.
The main idea in the proof of Theorem~\ref{continuity} is to use $\nu$ as the reference direction and approximate $g(\frac{x}{\e})$ by piecewise continuous functions, where each continuous parts are ``projections" of $g$ on $\{x\cdot\nu=0\}$ (see further description on the approximation below).  For simplicity of the presentation, we will  prove the theorem in $\R^2$: at the end of the proof will describe the modification required for $\R^n$.

\subsection{Description of the perturbation of boundary data and a sketch of the proof}

First let us describe the main ideas in the proof. We begin by introducing several notations.
For notational simplicity and clarity in the proof, we assume
that $\nu =e_2$: we will explain in the paragraph below how to
modify the notations and the proof for $\nu\neq e_2$. Let us define
$$
\Omega_0:=\Pi_{\nu}(0) = \{ x\in\RR^2: -1 \leq x_2:=x\cdot e_2\leq 0\}
$$
and for $i=1,2$
$$\Omega_i:= \Pi_{\nu_i}(0) =\{x \in \RR^2: -1 \leq x\cdot \nu_i \leq 0\}. $$
Let us also define the family of functions
\begin{equation}\label{neumann:proj}
g_i(x_1, x_2)=g_i(x_1)= g(x_1, \delta(i-1)), \hbox{ where }
i=1,...,m:= [\frac{1}{\delta}]+1.
\end{equation} (see Figure 3). Then $g_i$ is a $1$-periodic function with respect to $x_1$.
\begin{figure}
\center{\epsfig{file=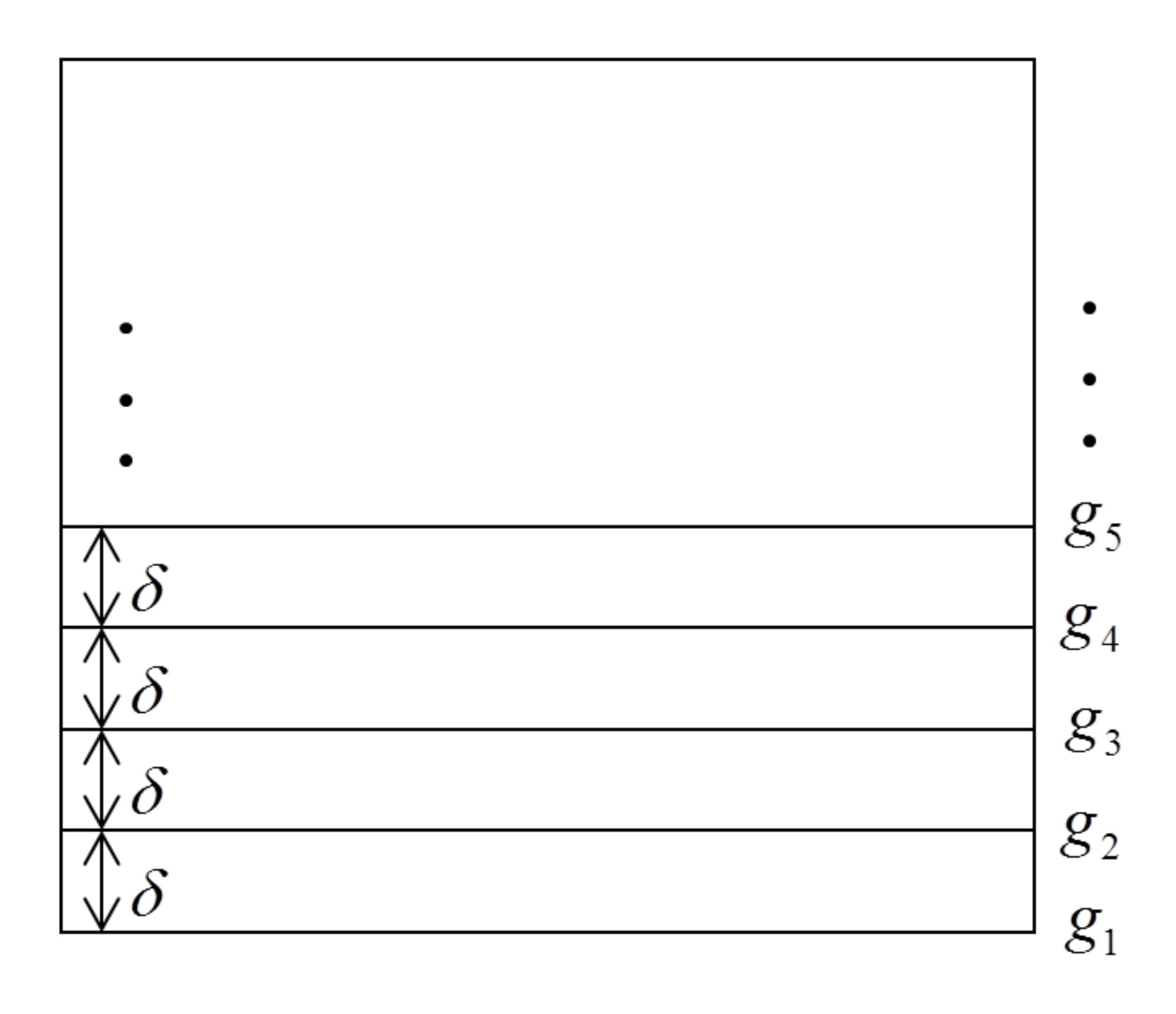,height=2.2in}} \center{Figure 3}
\end{figure}

\medskip

$\circ${\it A remark for $\nu\neq e_2$}:

\medskip

In two dimensions, if $\nu$ is a rational direction different from $e_2$, take the
smallest $K_\nu \in \NN$ such that $K_\nu \nu =0$ mod $\ZZ^2$. Then
we define $g_i(x) = g(x'+\delta(i-1)\nu)$, where $x' = x-x\cdot\nu$, and $g_i$ is a $K_\nu$-periodic function. If $\nu$ is an irrational direction,
take the  smallest $K_\nu \in \NN$ such that $|K_\nu \nu| \leq
\delta$ mod $\ZZ^2$. 
Then $g_i$ as defined above is almost $K_\nu$- periodic up to the
order of $\delta$ with respect to $x'$. We point out that it
does not make any difference in the proof divided in the following two subsections if we replace the
periodicity of $g_i$ by the fact that $g_i$'s are periodic up to the
order $\delta$. 

\bigskip

$\circ$ {\bf Proof by heuristics}

\medskip

Since the domains $\Omega_1$ and $\Omega_2$ point toward different
directions $\nu_1$ and $\nu_2$, we cannot directly compare their
boundary data, even if $\partial \Omega_1$ and $\partial \Omega_2$
cover most part of the unit cell in $\R^n/\Z^n$. To overcome this
difficulty we perform a multi-scale homogenization as follows.

\medskip

First we consider the functions $g_i$ $(i=1,..,m)$, whose profiles
cover most values of $g$ up to the order of $\delta^\beta$, where
$\beta$ is the H\"{o}lder exponent of $g$. Note that most values of
$g$ are taken on $\partial \Omega_1$ and on $\partial \Omega_2$
since $\nu_1$ and $\nu_2$ are irrational directions. On the other
hand, since $\nu_1$ and $\nu_2$ are very close to $e_2$ which is
a rational direction, the averaging behavior of a solution $u_\e$
in $\Omega_1$ (or $\Omega_2$) would appear only after $\e$ gets
very small, as $\nu_1$ (or $\nu_2$) approaches $\nu=e_2$.

\medskip

 Let $N=[\delta/|\nu_1-\nu|]$. If $|\nu_1-\nu|=|\nu_1-e_2|$ is chosen much
smaller than $\delta$,   then  we can say that the Neumann data
$g_1(\cdot/\e)$ is (almost) repeated $N$ times on $\Gamma_0=\{x\cdot\nu_1=0\}$ with period $\e$, up to the error $O(\delta^\beta)$.
(See Figure 4.) Similarly, on the next piece
of the boundary, $g_2(\cdot/\e)$ is (almost) repeated $N$ times
and then $g_3(\cdot/\e)$ is repeated $N$ times: this pattern will
repeat with $g_k$ ($k \in \NN$  mod $m$).
\begin{figure}
\center{\epsfig{file=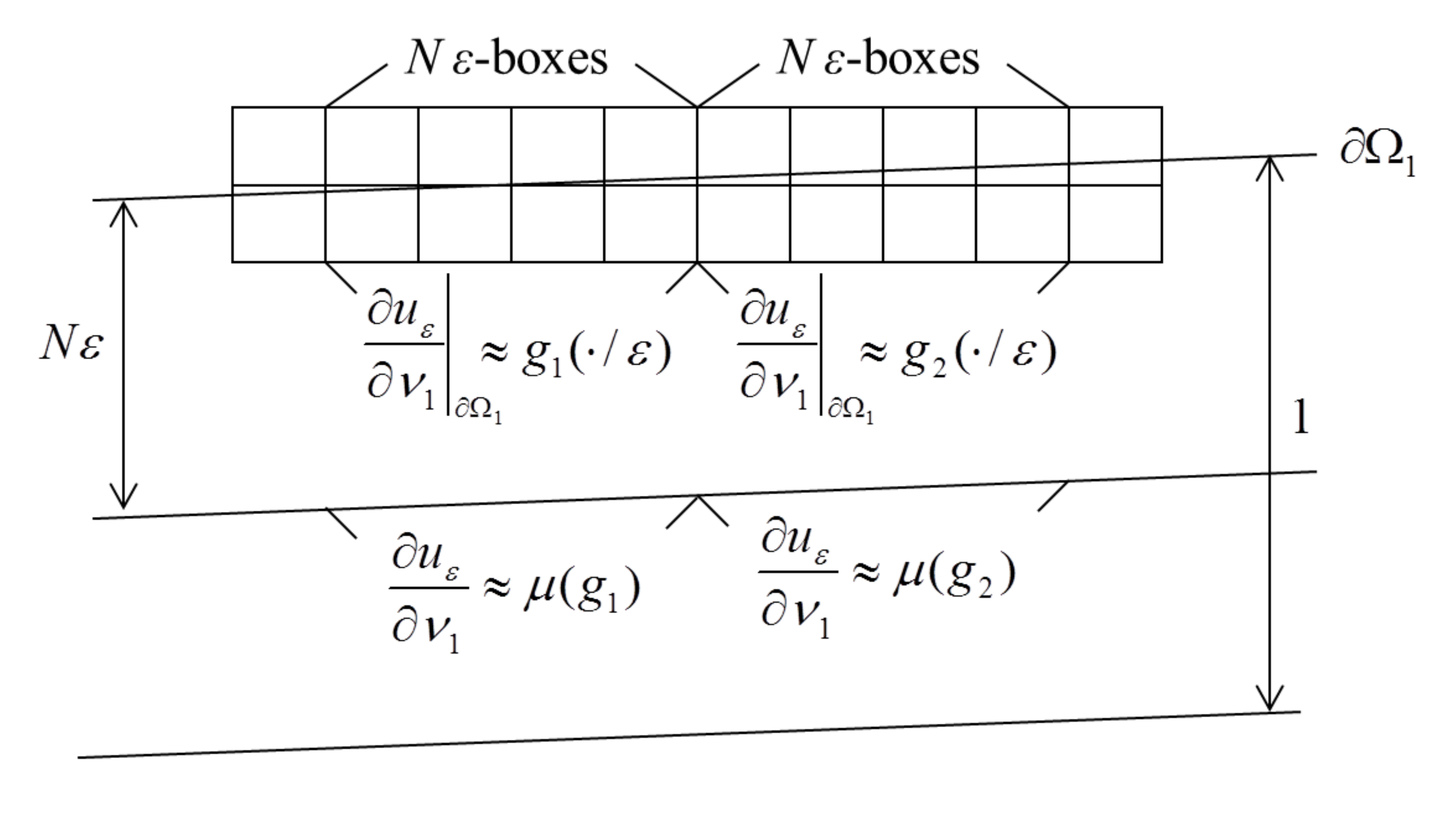,height=2.3in}} \center{Figure 4}
\end{figure}

\medskip

If $N$ is sufficiently large, i.e., if $|\nu_1-\nu|$ is sufficiently
small compared to $\delta$, the solution $u_\e$ of $(P_\e^{\nu_1})$ in $\Omega_1$ will
exhibit averaging behavior, $N\e$-away from $\Gamma_0$.
More precisely, on the hyperplane $H$ located $N\eta$-away from
$\Gamma_0$, $u_\e$ would be homogenized by the repeating
profiles of $g_i$ (for some fixed $i$) with an error of
$O(\delta^\beta)$. This is the first homogenization of $u_\e$ near
the boundary of $\Omega_1$: we denote, by $\mu(g_i)$, the
corresponding values of the homogenized slopes of $u_\e$ on $H$ in each $N\e$-segment.

\medskip

Now more than $N\e$ away from $\partial \Omega_1$, we obtain the
second homogenization of $u_\e$, whose slope is determined by
$\mu(g_i)$, $i=1,..,m$.  In the proof this second homogenization is divided into two parts, in the middle region which is $N\e$ to $KN\e$-away from $\Gamma_0$ and then in the rest of the domain $\Omega_1$. The homogenization argument in the middle region is to ensure that the oscillation of the operator $F$ in $x$-variable does not alter the behavior of the solution too much as $\nu$ varies. 
 Note that, due to \eqref{new} after rotation and reflection we may assume that the arrangement of $\mu(g_1),...\mu(g_m)$ is the same for $\nu_1$ and $\nu_2$. Therefore the second homogenization procedure applied to $\nu_1$ and $\nu_2$ yields that
$|\mu(\nu_1)-\mu(\nu_2)|$ is small.

\vspace{10pt}

Below we will present a rigorous proof for Theorem~\ref{continuity}, with above heuristics in mind.

\subsection{Estimates on localized barriers}

Let $\delta>0$ be given. We will first prove a series of lemmas which analyzes localized versions of $u^\e$ in the three areas  of the domain (near-strip, middle-strip, and inner-strip). Based on the estimates, in Section 4.3 we will then put together the localized barriers to construct appropriate test functions (sub- and supersolutions) of $(P_\e^{\nu_i})$.  The goal is, to compare these test functions with $u_\e^{\nu_i}$ to show the following: for  given $\delta>0$ and for $\nu_i$ and $\e$ satisfying the assumptions given in Theorem~\ref{continuity}, there exists a constant $\mu_0$ such that 
$$
|\mu(u_\e^{\nu_i}) - \mu_0| \leq \delta\hbox{ for } i=1,2.
$$
From above inequality Theorem~\ref{continuity} follows.

\vspace{10pt}

Let $\Omega_i$ and $g_i$ as given above. We continue with the notations. Given $\delta>0$, let $N_0$ be as given in Theorem~\ref{ext}, for Neumann boundary data $f\in C^{\beta}$ with $\max |f|\leq \max |g|$.  

\medskip

Let us  choose irrational unit vectors $\nu_1,
\nu_2 \in \RR^2$ such that
\begin{equation}\label{parameter}
0<\eta_0 \leq \e_0 \leq \min[\delta^{20}, \delta N_0^{-1}],
\end{equation}
 where $\e_0=|\nu_1-e_2|$ and $\eta_0=|\nu_2-e_2|$.
  Let us also define
\begin{equation}\label{steps}
N =[\dfrac{\delta}{\e_0}], \,\,\, M
=[\dfrac{\delta}{\eta_0}].
\end{equation}
and lastly suppose
\begin{equation}\label{order202}
0<\e\leq\delta\eta_0^{21/20}.
\end{equation}
The condition \eqref{order202} is to ensure that $2mN\e <<\delta$, i.e. we consider sufficiently small $\e$  so that $\mu(u_\e^{\nu_i})$ to approach the correct averaged slope $\mu(\nu_i)$.

\medskip

With above definition of $\e$ and $N$, consider the strip regions
$$
I_k=[(k-1)N\e,kN\e] \times \RR \,\,\hbox{  for }\,\, k \in \ZZ.
$$
  Let
  $\bar{k}\in [1,m]$ denote $k$ in modulo $m$, where $m=\displaystyle{[\frac{1}{\delta}]}+1$. Note that, since
$|\nu_1-e_2| <<\delta$, the function $g_{\bar{k}}(x_1/\e)$ defined in \eqref{neumann:proj} is (almost)
repeated $N$ times on $I_k \cap \Gamma_N$, where $\Gamma_N:=\{x\cdot\nu_1=0\}$. This fact and the H\"{o}lder continuity of $g$ yield
that
\begin{equation}\label{eqn1}
|g(\frac{x_1}{\e}, \frac{x_2}{\e}) - g_{\bar{k}} (\frac{x_1}{\e})| <
C\delta^\beta \hbox{ on } \Gamma_1 \cap I_k \quad\hbox{ for
} k\in\ZZ.
\end{equation}

Similarly one can define $\tilde{I}_k:= [(k-1)M\e, kM\e]\times\R$ for $k\in\Z$ to approximate $g$ on $\{x\cdot\nu_2=0\}$.

\medskip

\subsubsection{Estimates for solutions near the Neumann boundary}

\vspace{10pt}

Let us first consider the averaging behavior of solutions near ($N\e$- close to) the Neumann boundary.
For the strip domain $\Pi_\e:= \{-2N\e\leq x\cdot\nu_1 \leq 0\}$, let $u_\e$ solve
$$
\left\{\begin{array}{lll} F(D^2 u_\e, x/\e)=0
&\hbox{ in }& \{-2N\e \leq x\cdot \nu_1 \leq 0\}  \\ \\
\dfrac{\partial u_{\e}}{\partial\nu_1}(x) = g
(\dfrac{x}{\e}) &\hbox{ on }& \{x\cdot\nu_1=0\}\\ \\
u_{\e} = 1 &\hbox{ on } & \{x\cdot\nu_1=-2N\e\}
\end{array}\right.\leqno (N)
$$ 
 Similarly we define $\tilde{u}_{\e}$ in the strip domain
 $$
 \widetilde{\Pi_\e}:=\{-2M\e \leq x\cdot\nu_2\leq 0\}.
 $$

\medskip

First we replace the Neumann data $g$ with the family of functions $g_k$ introduced in \eqref{neumann:proj}.

Let $w_{\e}$ and $\tilde{w}_\e$ be solutions in $\Pi_\e$ and $\widetilde{\Pi_\e}$,
satisfying
$$
\left\{\begin{array}{lll} F(D^2 w_\e, x/\e)=0
&\hbox{ in }& \{-2N\e \leq x\cdot \nu_1 \leq 0\}  \\ \\
\dfrac{\partial w_{\e}}{\partial\nu_1}(x) = g_{\bar{k}}
(\dfrac{x_1}{\e}) &\hbox{ on }& \{x\cdot\nu_1=0\}\cap I_k
\quad\hbox{ for
} k\in\ZZ\\ \\
w_{\e} = 1 &\hbox{ on } & \{x\cdot\nu_1=-2N\e\}
\end{array}\right.
$$
and
$$
\left\{\begin{array}{lll} F(D^2 \tilde{w}_\e, x/\e)=0
&\hbox{ in }& \{-2M\e \leq x\cdot \nu_2 \leq 0\}  \\ \\
\dfrac{\partial \tilde{w}_{\e}}{\partial\nu_2}(x) = g_{\bar{k}}
(\dfrac{x_1}{\e}) &\hbox{ on }& \{x\cdot\nu_2=0\} \cap \tilde{I}_k
\quad\hbox{ for
} k\in\ZZ\\ \\
\tilde{w}_{\e} = 1 &\hbox{ on } & \{x \cdot \nu_2=-2M\e\}.
\end{array}\right.
$$

Next we localize the Neumann boundary data: for given $k\in\ZZ$, let $w_{\e, k}$ and $\tilde{w}_{\e
,k}$ solve
$$
\left\{\begin{array}{lll} F(D^2 w_{\e, k}, x/\e)=0
&\hbox{ in }& \{-2N\e \leq x\cdot \nu_1 \leq 0\}  \\ \\
\dfrac{\partial w_{\e, k}}{\partial\nu_1}(x) = g_{k}
(\dfrac{x_1}{\e}) &\hbox{ on }& \{x\cdot\nu_1=0\} \\ \\
w_{\e, k} = 1 &\hbox{ on } & \{x\cdot\nu_1=-2N\e\}
\end{array}\right.
$$
and
$$
\left\{\begin{array}{lll} 
F(D^2 \tilde{w}_{\e, k}, x/\e)=0
&\hbox{ in }& \{-2M\e \leq x\cdot \nu_2 \leq 0\}  \\ \\
\dfrac{\partial \tilde{w}_{\e, k}}{\partial\nu_2}(x) = g_{k}
(\dfrac{x_1}{\e}) &\hbox{ on }& \{x\cdot\nu_2=0\}\\ \\
\tilde{w}_{\e, k} = 1 &\hbox{ on } & \{x\cdot\nu_2=-2M\e\}
\end{array}\right.
$$

We will show that the profile of $u_\e$ and $\tilde{u_\e}$ does not change too much with above perturbation and localization of the Neumann data.

\begin{lemma} \label{a}
There exists a dimensional constant $C>0$ such that the following holds:
For each $k\in\Z^n$,
$$
|u_\e-w_\e|, \quad |w_{\e } - w_{\e, k}| \leq C\delta^\beta N \e \quad \hbox{ on } \{x\cdot\nu_1=-N\e\}\cap I_k
$$
 and
$$
|\tilde{u}_\e-\tilde{w}_\e|, \quad |\tilde{w}_{\e} -
\tilde{w}_{\e, k}| \leq C\delta^{\beta/2} M \e\quad\hbox{ on } \{x\cdot\nu_2=-M\e\}\cap \tilde{I}_k.
$$
\end{lemma}

\begin{proof} 

We will  only prove the lemma for $u_\e$ and $w_\e$. Let us define 
$$
\Gamma_D:=\{x\cdot\nu_1 = -2N\e\},\quad H:=\{x\cdot\nu_1=-N\e\}, \quad \Gamma_N:=\{x\cdot\nu_1=0\}.
$$

\medskip

1. Since $u_\e= w_\e$ on $ \Gamma_D$, and 
$$|\partial_{\nu_1} u_\e  -
\partial_{\nu_1} w_\e|\leq \delta^{\beta}
$$ on $\Gamma_0$, we obtain
that $|u_\e-w_\e| \leq 2\delta^{\beta} N\e$ by adding linear profiles 
$\pm\delta^{\beta}(x\cdot \nu_1+2N\e)$ to $u_\e$, and to $w_\e$ and then apply the comparison principle to get respective orders.

\medskip

2.  Next we compare $w_\e$ and $w_{\e k}$. Let $p \in I_k \cap
\Gamma_0$. Observe that 
$$
|g_k - g_{k+1}|\leq \delta ^\beta
$$
by the
construction of $g_k$ and by the H\"{o}lder continuity of $g$.
it follows that
\begin{equation}\label{conti:g}
|g_{k+l}-g_k| \leq \delta^{\beta/2} \hbox{ for } -\delta^{-\beta/2} \leq l \leq \delta^{-\beta/2}.
\end{equation}
Above inequality implies that
$$
|\partial_{\nu_1} w_\e -\partial_{\nu_1} w_{\e, k}|
\leq \delta^{\beta/2} \hbox{ in } B_{\delta^{-\beta/2} N\e }(p) \cap
\{x \cdot \nu_1 =0\}.
$$
Lastly note that, since $1\leq g\leq 2$, we have
$$
|w_\e -w_{\e, k} |\leq 2N\e \hbox{ on } \partial B_R(p)\cap \Pi_\e.
$$ 

We would like to apply the re-scaled (replacing $N\e$ as the unit scale) version of Lemma~\ref{localization} to $w_\e$ and $w_{\e ,k}$ in $B_R(p)\cap \Pi_\e$ with  $R=\delta^{-\beta/2}N\e $.  We then obtain
$$
|w_ \e - w_{\e,k}| \leq C\delta^{\beta/2}N\e \hbox{ on } H\cap I_k.
$$

\end{proof}

\medskip

Next we show that $w_{\e, k}$ and $\tilde{w}_{\e, k}$ are almost linear, and that they do not differ from each other very much. Let $\mu(w_{\e, k})$ be given as the approximating slope of 
 $w_{\e, k}$ as in \eqref{linear}, i.e. choose a point $p_k \in H\cap I_k$ and define $\phi_{\e,k}$ by
 $$
 \phi_{\e,k} = \mu(w_{\e,k}) (x\cdot\nu_1+2N\e)+ 1, \quad \mu(w_{\e,k}) = \frac{w_{\e,k}(p)-1}{-N\e}.
 $$

Similarly one can define $\tilde{\phi}_{\e,k}$ and $\mu(\tilde{w}_{\e, k})$.

\begin{lemma} \label{aa} Let $\alpha=\alpha(n,\Lambda,\lambda)$ be as given in Theorem~\ref{lemma-interior}, and let $\alpha_1 = \min[\alpha,\beta]$. Then there exists a dimensional constant $C>0$ such that the following holds true:
Respectively on $\{x\cdot \nu_1=-N\e\}$ and $\{x\cdot\nu_2=-M\e\}$,
\begin{equation}\label{flat}
|w_{\e, k} -\phi_{\e, k}|\leq CN\e\delta^{\alpha_1}\hbox{ and } \quad |\tilde{w}_{\e, k} -\tilde{\phi}_{\e, k}| \leq C M\e \delta^{\alpha_1}.
\end{equation}
Moreover
\begin{equation}\label{baby:continuity}
|\mu(w_{\e, k})-\mu(\tilde{w}_{\e, k})| =C\delta^{\alpha_1/2}.
\end{equation}
\end{lemma}

\begin{proof}
1. We will prove the first inequality on $w_{\e, k}$.  Fix $x_0\in H$.  Due to the fact that $|\nu_1-e_1|<\delta$,  the following holds: for any $y \in H$, there exists $z$ such that
\begin{equation} \label{lp}
 x_0=z \hbox{ mod } \e \ZZ^2,
\,\,\,\,\, |z_1-y_1| \leq \e, \,\,\,\,\, {\rm dist}(z, H) \leq
\delta \e.
\end{equation}
 Consider the function
$$
\psi_{\e, k}(x) =w_{\e, k}(x-(x_0-z)),
$$
which still solves $F(D^2\psi, \frac{x}{\e})=0$ in $\Pi_\e +(x_0-z)$.
We will compare $w_{\e, k}$
and
 $\psi_{\e, k}$ in the domain
 $$
 \Sigma:=\Pi_\e \cap (\Pi_\e + (x_0-z)).
 $$
Without loss of generality, we may assume
$$
\Sigma = \{-2N\e + \delta_0
\e \leq x \cdot \nu_1 \leq 0 \},
$$ where $0 \leq \delta_0 \leq |(x_0-z)\cdot \nu_1|\leq
\delta$.
  Observe that comparison with linear profiles yields that
 $$
  |w_{\e, 
  k} -\psi_{\e, k}| \leq \delta \e \hbox{ on }\{x \cdot \nu_1 =-2N\e +
  \delta_0\e\}.$$
Moreover, since $g_k(x)$ is constant in $x_2$-variable, the last property of (\ref{lp}) and Theorem~\ref{thm:reg2} yields that
  $$
  |\partial _{\nu_1} w_{\e, k}-\partial_{\nu_1} \psi_{\e, k}| \leq C \delta^{\alpha_1} \hbox{ on }\{x \cdot \nu_1 =0\}.
  $$ 
 
Let $h=C\delta^\alpha (x\cdot \nu_1 +2N\e) + \delta \e $. Then by
applying the comparison principle in $\Sigma$ we have 
$$
w_{\e, k} \leq \psi_{\e, k} +h \hbox{ and } \psi_{\e, k} \leq
w_{\e, k} +h\hbox{ in } \Sigma.
$$
 Since $0 \leq h\leq 2N \e \delta^\alpha$ on $H:=\{x\cdot \nu_1=-N\e\}$, we
get
$$ |w_{\e, k}(x_0) -w_{\e, k}(z)|= |w_{\e,k}(x_0) - \psi_{\e,k}(x_0)| \leq 2N \e \delta^{\alpha_1}\hbox{ on } H.$$
Now we conclude that 
\begin{eqnarray} \label{Claim22}
|w_{\e, k}(x_0)- w_{\e, k}(y)|&\leq& |w_{\e, k}(x_0)- w_{\e, k}(z)|+|w_{\e, 
k}(z)- w_{\e, k}(y)| \nonumber \\&\leq& 2N \e \delta^{\alpha_1} + |w_{\e, 
k}(z)- w_{\e, k}(y)|  \nonumber
\\
&\leq &  2N \e \delta^{\alpha_1} + C|z-y|
\nonumber\\
&\leq & 3N \e \delta^{\alpha_1},
 \end{eqnarray}
where the last inequality follows since $N \delta^\alpha \sim
\delta^{1+\alpha}/\e_0>1$.

\medskip

3. To prove the statement for $\tilde{w}_\e$, one can argue as above, by replacing $I_k$ by $\tilde{I}_k$ and $\Pi_\e$ by $\tilde{\Pi}_\e$, and $H$ with $\tilde{H}:= \{x\cdot\nu_2=-M\e\}$.

\medskip

4. We proceed to prove \eqref{baby:continuity}. Recall that $M>N$. First we compare $w_{\e,k}$ with $\rho_\e$, solving
$$
\left\{\begin{array}{lll}
F(D^2\rho_\e,\frac{x}{\e}) =0 &\hbox{ in }& \{-2N\e \leq x\cdot\nu_2\leq 0\},\\  \\
\partial_{\nu_2} \rho_\e = g_k(\frac{x}{\e}) & \hbox{ on } &\{x\cdot\nu_2=0\},\\ \\
\rho_\e = 1 &\hbox{ on }& \{x\cdot\nu_2 = -2N\e\}
\end{array}\right.
$$
We will show that 
\begin{equation}\label{first0}
|\rho_\e - w_{\e, k } | \leq C\delta^{\alpha_1/2} N\e \hbox{ in } \{-2N\e \leq x\cdot\nu_2\leq 0\}.
\end{equation}
If \eqref{first0} holds, then we can conclude by barrier arguments as given in the proof of Lemma~\ref{Claim 4} to show that $\rho_\e$ and $\tilde{\omega}_\e$ are close.

Due to \eqref{flat}, it is enough to argue in the $N\e$-neighborhood of the origin.
 Note that the Neumann boundary of $\rho_\e$ and $w_{\e, k}$ only differ by $R\delta\e$ in $RN\e$-neighborhood of the origin. Therefore  Theorem~\ref{thm:reg2} yields that
\begin{equation}\label{eqno1}
|\partial_{\nu_1}\rho_\e- g_k(\frac{x}{\e})| \leq (R\delta)^{\alpha_1} \hbox{ on }\{x\cdot\nu_1=0\} \cap \{|x|\leq RN\e\}.
\end{equation}

Moreover, since $1\leq g\leq 2$, we have $|\rho_\e- w_{\e,k}|\leq 2N\e$ on $\{x\cdot\nu_1 = -2N\e\}$. Therefore \eqref{eqno1} yields that
\begin{equation}\label{eqno2}
|\rho_\e- w_{\e,k} | \leq  C(R\delta)^{\alpha_1} N\e \quad\hbox{ on }\{ x\cdot\nu_1 = -N\e\}\cap I_{k+j}, |j|\leq R.
\end{equation}

Now due to \eqref{eqno1} and \eqref{eqno2},  Lemma~\ref{localization} yields that 
$$
|\rho_\e- w_{\e, k} | \leq C((R\delta)^{\alpha_1}+R^{-2})N\e.
$$
Let us take $R=\delta^{-1/2}$ to conclude.

\end{proof}

The following Corollary is immediate from Lemma~\ref{a} and Lemma~\ref{aa} as well as Theorem~\ref{thm:reg2}.

\begin{corollary} \label{coco}
Denote $\mu_k:=\mu(w_{\e, k}) $, then for $\alpha_1$ given as in Lemma~\ref{aa}
$$|\frac{\partial }{\partial\nu_1}u_\e  - \mu_k|\leq \delta^{\alpha_1} \hbox{ on
} H \cap I_k.$$
\end{corollary}

\subsubsection{Estimates for solutions in the middle region}

From the previous section we have seen that the solutions $\tilde{u}_\e$ and $\tilde{u}_\e$ average in the unit of $N\e$, when the solution is $N\e$-away from the Neumann boundary, and moreover that they average to a similar value. If the operator $F$ was homogeneous we could use this result and perform a second homogenization, and use \eqref{new} and \eqref{baby:continuity} to conclude. However we have to be careful with the inhomogeneities of $F$, such that different normal directions of the hyperplanes pointing toward $\nu_1$ and $\nu_2$ does not make a difference in the way the respective solutions average out. This is what we are going to analyze in this subsection.

\vspace{10pt}

 Let $K=\frac{1}{\delta}$, and let us consider the domain 
 $$\Sigma=\{x\in \RR^2: -2KmN\e\leq
x \cdot \nu_1 \leq -N\e \} .
$$ 
Note that by \eqref{order202} $2KmN\e \leq 1$. Let $\mu_k$ as given in Corollary~\ref{coco}, and let $v_\e$ solve
$$
\left\{\begin{array}{lll} F(D^2 v_\e, \cdot/\e)=0
&\hbox{ in }& \Sigma \\ \\
\dfrac{\partial v_{\e}}{\partial\nu_1} = \mu_k &\hbox{ on }& H
\cap I_k \quad\hbox{ for
} k\in\ZZ\\ \\
v_{\e} = 1 &\hbox{ on } & \{x \cdot \nu_1=-2KmN\e\}.
\end{array}\right.
$$
Next let $\bar{v}_\e$ solve
$$
\left\{\begin{array}{lll} \bar{F}(D^2 \bar{v}_\e)=0
&\hbox{ in }& \Sigma \\ \\
\dfrac{\partial \bar{v}_{\e}}{\partial\nu_1} = \mu_k &\hbox{ on
}& H \cap I_k \quad\hbox{ for
} k\in\ZZ\\ \\
\bar{v}_{\e} = 1 &\hbox{ on } & \{ x\cdot \nu_1=-2KmN\e\}.
\end{array}\right.\leqno (\bar{P}^{\nu_1}_\e)
$$

We will show that $v_\e$ and $\bar{v}_\e$ are close to each other.
\begin{lemma} \label{lem35} 
There exists $\alpha = \alpha(n,\Lambda, \lambda)$ such that 
\begin{equation}\label{first}
|v_\e(x)-\bar{v}_\e(x)|\leq\delta^{\alpha} N\e \quad\hbox{ in
} \Sigma.
\end{equation}
\end{lemma}

\begin{proof} 
 Let us apply Theorem~\ref{ext} with $f(x)=\mu_k$ on $H\cap I_k$ (or its interpolation so that $f(N\e x)$ is H\"{o}lder continuous)  to the re-scaled function 
 $$
 u_N(x) = (N\e)^{-1}v_\e(\frac{x}{N\e}).
 $$ Due to \eqref{parameter}, $N \geq N_0$, and thus Theorem~\ref{ext} yields ~\eqref{first}.
\begin{equation}\label{first}
|v_\e(x)-\bar{v}_\e | \leq \delta N\e \hbox{ in } \Sigma_1.
\end{equation}

\end{proof}

Next we are going to show that $\bar{v}_\e$ is close to a linear profile away from $H$. To see this, recall that the slopes $\mu_1$,...,$\mu_m$ are repeated
$K$-times on the Neumann boundary of $\Sigma$. Hence the homogenization arguments for the homogenized operator $\bar{F}$ should apply in our setting, if $K$ is chosen sufficiently large. This is what we will show below in detail.

\vspace{10pt}
 
 Let $\mu(v_\e)$ and $\mu(\bar{v}_\e)$ be respectively the average slope of 
 $v_\e$ and $\bar{v}_\e$ respectively in $\Sigma$ as in \eqref{linear}.

\begin{lemma} \label{lem36}
Let $\alpha$ given as in Lemma~\ref{aa}. Then there exists a dimensional constant $C>0$ such that 
$$|\frac{\partial}{\partial\nu_1} \bar{v}_\e  - \mu(\bar{v}_\e)| \leq C\delta^{\alpha}\quad\hbox{ on }L=\{-2KmN\e \leq x\cdot\nu_1 \leq -KmN\e\} .$$

\end{lemma}

\begin{proof} 
Choose a point $x_0 \in L$.  Since $\bar{v}_\e$ has constant slope
$\mu_k$ on $H\cap I_k$ with $|H\cap I_k|=N\e$, and since the slopes
$\mu_1$,...,$\mu_m$ are repeated on $H$,  for any $y$ such that $(y-x_0)\cdot\nu_1=0$ there is a point $z \in L$ such that
\begin{equation}\label{r1}
 |y-z|\leq m N\e, \quad x_0-z \in H,
\end{equation}
and
\begin{equation}\label{r2}
\frac{\partial}{\partial\nu_1}\bar{v}_\e(x) =\frac{\partial}{\partial\nu_1}\bar{v}_\e(x-(x_0-z)) \hbox{ on } H.
\end{equation}

Due to \eqref{r2} and Theorem~\ref{thm:comp} we have $\bar{v}_\e(x)=\bar{v}_\e(x-(x_0-z))$. Moreover

\begin{equation} \label{Claim22}
|\bar{v}_\e(x_0)- \bar{v}_\e(y)|= |\bar{v}_\e(z)- \bar{v}_\e(y)|
\leq  CmN\e 
 \end{equation}

where the  inequality follows from the interior Lipschitz regularity of $\frac{1}{KmN\e}\bar{v}_\e(KmN\e x)$ (Theorem~\ref{lemma-interior}).  With (\ref{Claim22}), we can conclude due to the $C^{1,\alpha}$ regularity of $\frac{1}{KmN\e}\bar{v}_\e(KmN\e x)$.
 \end{proof}

Similarly, for the normal direction $\nu_2$ one can construct $\tilde{v}_{\e}$ and $\tilde{\bar{v}}_{\e}$ accordingly in 
$$
\tilde{\Sigma}=\{-2KmM\e \leq x\cdot \nu_2 \leq -M\e\}.
$$

Note that Corollary~\ref{coco} implies
\begin{equation} \label{37}
|\mu(w_{\e,k})-\mu(\tilde{w}_{\e,k})| =\delta^{\alpha}.
\end{equation}
Note that, due to \eqref{new} after rotation and reflection we may assume that the arrangement of $\mu(g_i)$ is the same for $\nu_1$ and $\nu_2$. Therefore from \eqref{37} and \eqref{new} it follows that 
\begin{equation}\label{conti:direction}
|\mu(v_\e)-\mu(\tilde{v}_{\e})| \leq \delta^{\alpha.}
\end{equation}

Due to Lemma~\ref{lem35} there exists $\alpha = \alpha(n,\lambda, \Lambda, \beta)$ such that 
\begin{equation}\label{38}
|\mu(\bar{v}_\e)-\mu(\tilde{\bar{v}}_{\e})| \leq \delta^{\alpha}.
\end{equation}

\subsection{The proof of the main theorem}

Recall that $u_\e^{\nu_i}$ for $i=1,2$ solve  $(P_\e^{\nu_i})$ in $\Pi_{\nu}(0)$.  With the estimates obtained in the previous section, we are now ready to construct our barrier for $u_\e^{\nu_1}$ and $u_\e^{\nu_2}$. As mentioned before we will construct barriers in three separate regions and patch them up. Let us construct it for the normal direction $\nu_1$: let $\alpha_0=\min (\alpha, \beta/4)$ where $\alpha$ is as given in \eqref{38}.

\medskip

{\bf In far-away region:}  let us define
$$
f_\e(x):= \Lambda(x\cdot\nu_1 +1)+1,\leqno(F)
$$
where
$$
\Lambda = \mu(\bar{v}^\e)+10\delta^{\alpha_0}.
$$
{\bf In the middle strip:} next consider $\rho_\e$: the unique (bounded) viscosity solution of
$$
\left\{\begin{array}{lll}
F(D^2\rho_\e, \frac{x}{\e})=0 &\hbox{ in } & \{-KmN\e\leq x\cdot\nu_1 \leq -2N\e\};\\ \\
\frac{\partial\rho_\e}{\partial\nu_1}= \Lambda_2 &\hbox{ on }& H=\{x\cdot\nu_1=-2N\e\};\\ \\
\rho_\e = f_\e &\hbox{ on } & \{x\cdot\nu_1=-KmN\e\}
\end{array}\right.\leqno(M)
$$
where $\Lambda_2(x)$ is obtained by approximating $\mu_k$ such that

\begin{equation}\label{condition_a}
\Lambda_2(x)\in C^1(\R^n)\hbox{ with its } C^1 \hbox{ norm less than }\delta^{\beta}(N\e)^{-1};
\end{equation}
and
\begin{equation}\label{condition_b} 
 \Lambda_2 \hbox{ is periodic with period } mN\e \hbox{ on } H, \hbox{  and  }\mu_k+2\delta^{\beta}\leq\Lambda_2(x) \leq \mu_k +5\delta^{\beta}.
\end{equation}
Such $\Lambda_2$ satisfying \eqref{condition_a}-\eqref{condition_b} exists  for $\alpha_0<\beta/4$ since, due to \eqref{conti:g}, for any given $k,l\in\ZZ$, 

\begin{equation}\label{data:flat}
 |\mu_{k+l}-\mu_{k}| \leq |l|\delta^{\beta}.
 \end{equation}
 
{\bf In the near-boundary region:} last, let $\phi_\e$ be the unique (bounded) viscosity solution of 
$$
\left\{\begin{array}{lll}
F(D^2\phi_\e, \frac{x}{\e})=0 &\hbox{ in } & \{-2N\e\leq x\cdot\nu_1 \leq -2N\e\};\\ \\
\frac{\partial \phi_\e}{\partial \nu_1} = g(\frac{x}{\e}) &\hbox{ on } & \{x\cdot\nu_1=0\}\\ \\
\phi_\e = \rho_\e &\hbox{ on }& \{x\cdot\nu_1=-2N\e\};\\
 \end{array}\right. \leqno (N)
$$

\medskip

Our goal is  to show that the function
$$
U_\e:= \left\{\begin{array}{ll} 
                    f_\e &\hbox{ in } \{-1 \leq x\cdot\nu_1\leq -KmN\e\}\\ \\ 
                    \rho_\e &\hbox{ in } \{-KmN\e \leq x\cdot\nu_1\leq -2N\e\}\\ \\
                     \phi_\e &\hbox{ in } \{-2N\e \leq x\cdot\nu_1 \leq 0\} 
                     \end{array}\right. 
$$
is a supersolution of $(P_\e)$ in $\Pi_{\nu_1}$. 

\medskip

Observe that, due to Lemma ~\ref{lem35} and Lemma~\ref{lem36},

\begin{equation}\label{patching1}
\frac{\partial\rho_\e}{\partial \nu_1} \leq \Lambda \hbox{ on } \{x\cdot\nu_1=-KmN\e\}.
\end{equation}
This means the patch-up of $f_\e$ and $\rho_\e$ together is a supersolution, i.e., 
$$
F(D^2U_\e, \frac{x}{\e}) \geq 0 \hbox{ in } \{-1\leq x\cdot\nu_1 \leq -KmN\e\}.
$$

It remains to see if the same holds for the patch-up of $\rho_\e$ and $\phi_\e$. 

\medskip

Note that $\rho_\e$ is a constant on its Dirichlet boundary $\{x\cdot\nu_1=-KmN\e\}$.  Due to  this fact and the small oscillation of $\Lambda_2$ given in \eqref{condition_a}, as well as Theorem~\ref{thm:reg2} and Theorem~\ref{ext} applied to re-scaled versions of $\rho_\e$  yields that  $\frac{1}{N\e}\rho_\e(N\e x)$ has its $C^{1,\beta}$ norm of size $\delta^{\beta/2}$ in the scale of $N\e$ . More precisely, for any $y\in H:=\{x\cdot\nu_1=-2N\e\}$ we have

\begin{equation}\label{sol:flat} 
 |\rho_\e(x)-\rho_\e(y) - \eta_0\cdot(x-x_0) | \leq \delta^{\alpha_0}N\e  \hbox{ on } H \cap B_{2\delta^{-\alpha_0}N\e}(y),
 \end{equation}
 where $\eta_0 = D\rho_\e(x_0)-D\rho_\e(x_0)\cdot\nu_1.$

\medskip

Let us fix a point $x_0\in H$. Next we will invoke Lemma~\ref{localization} for $\frac{1}{N\e}\phi_\e(N\e x)$ with $A=\delta^{\alpha_0}$ and $R=2\delta^{-\alpha_0}$ to obtain 
\begin{equation}\label{kookoo}
|\phi_\e - \eta_0\cdot(x-x_0)-\bar{\phi_\e}| \leq \delta^{\alpha_0}N\e \hbox{ in }\{-2N\e\leq x\cdot\nu_1 \leq 0\}\cap B_{\delta^{-\alpha_0}N\e}(0),
\end{equation}
where $\bar{\phi_\e}$ solves $(N)$ with the fixed boundary data $\rho_\e$ replaced by the constant $\rho_\e(x_0)$.

\medskip

 Suppose $x_0\in H\cap I_k$. Due to Corollary~\ref{coco} we have 

\begin{equation}\label{kookoo22}
|\mu(\bar{\phi}_\e)- \mu_k|< \delta^{\alpha_0} N\e\hbox{ on } H:=\{x\cdot\nu_1=-N\e\}.
\end{equation}

\eqref{kookoo} as well as \eqref{kookoo22} yields that

 \begin{equation}\label{kookoo3}
|\phi_\e(x)- (\eta_0\cdot(x-x_0)+\mu_k(x-x_0)\cdot\nu_1 -\bar{\phi}(x_0)|< 2\delta^{\alpha_0}N\e \hbox{ on } \{-2N\e \leq x\cdot\nu_1\leq-N\e\}.
\end{equation}

It follows from \eqref{sol:flat},\eqref{kookoo3} and the $C^{1,\alpha}$ regularity of $\rho_\e$ and $\phi_\e$ in $N\e$-scale  that
$$
\frac{\partial}{\partial\nu_1} \phi_\e \leq \frac{\partial}{\partial\nu_1} \rho_\e \hbox{ on } H,
$$

and thus the patch-up of $\rho_\e$ and $\phi_\e$ is a supersolution, i.e.
$$
F(DU^\e, \frac{x}{\e}) \geq 0 \hbox{ in } \{ -KmN\e \leq x\cdot\nu_1\leq -N\e\}.
$$
\medskip

Summarizing, we have shown that $U_\e$ is a supersolution of $(P^{\nu_1}_\e)$ in $\Pi_{\nu_1}(0)$.  Hence by comparison principle we obtain that  $u_\e^{\nu_1} \leq U_\e$ in $\Pi_{\nu_1}(0)$.  In particular,
$$
  u^{\nu_1}_\e(x) \leq U_\e(x) \leq (\mu(v^\e)+10\delta^{\alpha_0})(x\cdot\nu_1-1)+1.
$$

Similarly one can construct a subsolution $V_\e$, to show that 
\begin{equation}
|u^{\nu_1}_\e (x)- \mu(v^\e)(x\cdot\nu_1-1)-1| \leq 10\delta^{\alpha_0}.
\end{equation}

Above equation yields that  
$$
|\mu(u^{\nu_1}_\e) - \mu(\bar{v}^\e)| \leq 10\delta^{\alpha_0}.
$$ Similarly, 
$$
|\mu(u^{\nu_2}_\e)-\mu(\bar{\tilde{v}}^\e) | \leq 10\delta^{\alpha_0}.
$$
Moreover by \eqref{38} we have 
$$
|\mu(\bar{v}^\e) - \mu(\bar{\tilde{v}}^\e)| \leq \delta^\alpha_1.
$$
This proves Theorem~\ref{continuity} (a).

\medskip

Lastly observe that $(\bar{P}_\e^{\nu_1})$  given in section 4.2.2,  has its only microscopic term in the Neumann boundary data, which are constants in each segments $I_k$ of $H$ of size 
$N\e \leq \delta \e^{1/20}
$ due to \eqref{order202}, with the constants repeating itself after $\frac{1}{\delta}$ pieces of $I_k$: that is the Neumann boundary data is periodic with period $\e^{1/20}$. Since $\bar{v}_\e$ solves $(\bar{P}_\e^{\nu_1})$,  $\mu(\bar{v}_\e)$ converges uniformly to a limit $\mu_0$ as $\e\to 0$ with the convergence rate only depending on $\delta$ and $\e$, as long as \eqref{order202} is satisfied. Therefore, if necessary approximating the Neumann data with functions of the form $f(\frac{x}{\e^{1/20}})$ where $f$ is periodic and Lipschitz,  we conclude that
\begin{equation}\label{conclusion}
|\mu(u^{\nu_i}_\e) - \mu_0 | \leq \delta^{\alpha_0}\hbox{ for } i=1,2,
\end{equation}
if $\e$ is sufficiently small.  Since $\delta>0$ is arbitrarily chosen and $\alpha_0$ only depends on $\beta,n,\lambda$ and $\Lambda$, this proves the second claim of Theorem~\ref{continuity} (b).

\hfill$\Box$

\begin{remark} [For dimensions higher than two]
In higher dimensions $n>2$ and for $\nu=e_n$,
one can define
$$
g_i(x_1,...,x_{n-1},x_n) = g_i(x_1,...,x_{n-1}) = g(x_1,...,x_{n-1}, \delta(i-1))
$$
for $i=0,1,...,m=[\delta^{-1}]$. Let us also define
$$
I_{k_1,k_2,...,k_{n-1}}= [(k_1-1)N\e, k_1N\e] \times...\times [(k_{n-1}-1)N\e,k_{n-1}N\e]\times \R.
$$
Then parallel arguments in the previous two subsections apply to yield the corresponding result to Theorem~\ref{continuity}.
\end{remark}

\section{In general domain}
In this section we will use Theorem~\ref{continuity} as well as stability properties of viscosity solutions to prove our main result.  As given in the introduction, let $\Omega$ be a bounded domain with $C^2$ boundary containing a unit ball $K=B_1(0)$. Let us denote $\nu=\nu_{x}$ the outward normal vector of $\Omega$ at $x\in\partial\Omega$.
Suppose that $\partial\Omega$ does not have any flat boundary parts in the following sense:
For any $x_0\in\partial\Omega$ and sufficiently small $\sigma$, there exists $r_0>0$ and $r(\sigma)>0$ such that 
\begin{equation}\label{condition22}
|\nu_x-\nu_{x_0}| \geq \sigma \hbox{ if } x\in \partial\Omega\cap\{ r(\sigma)< |x-x_0| <r_0\}.\end{equation}

Note that, for example, any strictly convex domain $\Omega$ satisfies (a)-(b). We now state the main theorem.

\begin{theorem}\label{main2}
Let $\Omega$ and $K$ as given above, and let
$\bar{\mu}:\mathcal{S}^n\to [1,2]$ be as given in Theorem~\ref{continuity}.
Consider $u_\e$ solving
$$
\left\{\begin{array}{lll}
F(D^2u_\e,\frac{x}{\e})=0 &\hbox{ in } &\Omega-K;\\ \\
\frac{\partial u_\e}{\partial \nu}=g(\frac{x}{\e}) &\hbox{ on } &\partial\Omega;\\ \\
u=1 &\hbox{ on } & K,
\end{array}\right.
$$
and let $\bar{F}$ be the homogenized operator given in Theorem~\ref{original_convergence} with the assumption \eqref{new}. Then $u^\e$ converges uniformly to the unique viscosity solution $u$ of
$$
\left\{\begin{array}{lll}
\bar{F}(D^2u)=0 &\hbox{ in } &\Omega-K;\\ \\
\frac{\partial u}{\partial \nu}=\bar{\mu}(\nu) &\hbox{ on } &\partial\Omega;\\ \\
u=1 &\hbox{ on } & K.
\end{array}\right.\leqno (P)
$$
\end{theorem}

\begin{remark}
Note that the uniqueness of $u$ follows from the continuity of $\bar{\mu}(\nu)$ and  Theorem~\ref{thm:comp}.
\end{remark}

Next we locally approximate $u_\e$ with the solutions associated with strip domains discussed in section 4, based on the regularity properties of $u_\e$ and $\partial\Omega$.
 
 \medskip
 
For given domain $\Omega$ , $p\in\partial\Omega$ and $0<k<1$, let us define $\nu_0=\nu_p$ and 
$$
\Sigma_k:=\Omega\cap\{x: -\e^k < (x-p)\cdot\nu_0\} \cap B_{\e^{5k/8}}(p).
$$

Let $w_\e$ solve
$$
\left\{\begin{array}{lll}
F(D^2w_\e, \frac{x}{e})=0 &\hbox{ in } & \Sigma_k\\ \\
\displaystyle{\frac{\partial w_\e}{\partial\nu}=g(\frac{x}{\e})} &\hbox{ on } &\partial\Omega\cap\partial {\Sigma}_k\\ \\
w_\e=1 &\hbox{ on }& \partial {\Sigma}_k - \partial\Omega.
\end{array}\right.
$$
Next let 
$$
\widetilde{\Sigma}_k:= \{ -\e^k < (x-p)\cdot\nu_0 <0\}\cap B_{\e^{5k/8}}(p),
$$ and let $v_\e$ solve
$$
\left\{\begin{array}{lll}
F(D^2 v_\e, \frac{x}{\e})=0 &\hbox{ in }& \tilde{\Sigma}_k;\\ \\
\displaystyle{\frac{\partial v_\e}{\partial\nu_0} =g(\frac{x}{\e})} &\hbox{ on }& \{(x-p)\cdot\nu_0=0\};\\ \\
v_\e=1 &\hbox{ on }&  \{-\e^k = \nu\cdot(x-p)\}.
\end{array}\right.
$$
Since $\partial\Omega$ is $C^2$, we may assume that the hyperplane $\{(x-p)\cdot\nu_0=0\}$  is contained in
the $\e^{5k/4}$-neighborhood of $\partial \Omega$ in
$B_{\e^{5k/8}}(p)$.

\begin{lemma}\label{close}
There exists $0<k<1$ and $a>0$: independent of $\e$,$p$ and $\nu_0$ such that 
$$
|w_\e - v_\e| \leq \e^{k+a} 
$$
in $\Sigma_k \cap B_{\e^{2k/3}}(p)$.
\end{lemma}
\begin{proof}

After a translation, we may assume that $p=0$. First note that $w_\e$ and $v_\e$ will
oscillate at most of order $\e^{k}$ in their respective domains
$\Sigma_k$ and $\tilde{\Sigma}_k$: This can be checked by comparison with linear
profiles. Let us consider the re-scaled functions
$$
\tilde{w}(x)=  w_\e(\e x)/\e \,\,\hbox{ and }\,\,\tilde{v}(x)=
v_\e(\e x)/\e.
$$
Then Theorem~\ref{thm:reg2}, the H\"{o}lder continuity of $g$, and the fact that $\tilde{w}$ and $\tilde{v}$ oscillates up to $C\e^{k-1}$ yields 
$$
\|\tilde{w}\|_{C^{1,\alpha}}, \|\tilde{v}\|_{C^{1,\alpha}} \leq C\e^{k-1}
$$
 in their respective domains  $\frac{1}{\e}\Sigma_k$ and $\frac{1}{\e}\tilde{\Sigma}_k$, where $0<\alpha<1$ is as given in Theorem~\ref{thm:reg2}.
Consequently we have
 \begin{equation}\label{equation111}
 |\partial_{\nu_0}w_\e-\partial_{\nu_0}v_{\e}|\leq  O(\e^{k-1+(\frac{5k}{4}-1)\alpha})\hbox{ on } H:=\{x\cdot\nu_0 = -\e^{5k/4} \}\cap B_{\e^{5k/8}}.
 \end{equation}
   Let us choose $k$ sufficiently close to $1$ so that 
   $$
   k-1+(\frac{5k}{4}-1)\alpha > \alpha/6.
   $$
Let us define 
$$
h(x) :=  \e^{-k/4}|x-p|^2
+n\frac{\Lambda}{\lambda}\e^{-k/4}(\e^{2k}-|(x-p)\cdot \nu_0|^2),
$$
so that 
$$
-\mathcal{P}^-(D^2h) >0, \quad h\geq \e^k\hbox{ on } \overline{\Sigma_k}\cap \partial B_{5k/8}, \quad h \geq 0\hbox{ on } \{(x-p)\cdot\nu_0=-\e^k\}
$$
and $\partial_{\nu_0} h  \sim O(\e^k) <<\e^{\alpha/6} $ on $H$,
where $H$ is given in \eqref{equation111}.
 Now by \eqref{equation111},  the comparison principle (Theorem~\ref{thm:comp}) applies to $\omega_\e-v_\e$ and $h$ in the domain $\Sigma_k\cap \{x\cdot\nu_0 \leq -\e^{5k/4}\}$  to yield

\begin{equation}\label{upperb}
|w_\e(x) - v_\e(x)| \leq C\e^{\alpha/6}((x-p)\cdot\nu_0+\e^{k}) + h (x)
\hbox{ in } \Sigma_k.
\end{equation}
By evaluating the upper
bound obtained in \eqref{upperb} in the region $\Sigma_k\cap B_{\e^{2k/3}}$, we conclude the lemma for
$a = \min\{ \alpha/6, k/12\}$.

\medskip

\end{proof}

We are now ready to  show the main proposition. Let us define the
semi-continuous limits 
$$
{\rm lim\,sup}^*  u^\e (x) := \liminf_{\e\to 0}( \sup\{u^\e(y):
y\in\bar{\Omega} \hbox{ and } |x-y|\leq \e\})
$$
and
$$
{\rm lim\,inf}_* u^\e(x):= \limsup_{\e\to 0} (\inf\{u^\e(y):
y\in\bar{\Omega}  \hbox{ and } |x-y| \leq \e\}).
$$

\medskip

It is straightforward from the definition to check that ${\rm lim\,sup}^*  u^\e$ is upper semicontinuous and ${\rm lim\,inf}_*u^\e$ is lower semicontinuous.

\medskip

Our main theorem is a consequence of the following proposition.

\begin{proposition}\label{prop:main}
The following holds true:
\begin{itemize}
\item[(a)] $\bar{u}:=\limsup^* u^\e$ is the viscosity subsolution of $(P)$;
\item[(b)] $\underline{u}:= \liminf_* u^\e$ is the viscosity supersolution of  $(P)$.
\end{itemize}
\end{proposition}

\medskip

\begin{proof}

We will only show (a): parallel arguments apply to (b).  The proof follows the perturbed test function method introduced by Evans \cite{E}, and consists of several round of perturbations of the test functions $\phi$. 

\vspace{10pt}

1.
It follows from previously known results (see e.g., \cite{CS},\cite{CSW}) that
$$
\bar{F}(D^2\bar{u})\leq 0\hbox{ in }\Omega.
$$ 
Moreover from comparison with linear functions it is
straightforward to show that $\bar{u} \leq 1$ on $K$.
 Therefore if $\bar{u}$ fails to be a subsolution of $(P)$, then there exists a
smooth function $\phi$ which touches $\bar{u}$ from above at
a boundary point $x_0\in\partial\Omega$ and satisfies, for some $\delta>0$,
\begin{equation}\label{condition}
\bar{F}(D^2\phi) (x_0)>2\delta\quad \hbox{ and }\quad
\frac{\partial{\phi}}{\partial\nu} (x_0) \geq \bar{\mu}(\nu_{x_0}) +
3\delta
\end{equation}
Due to the regularity of  $\phi$ and the continuity of $\bar{\mu}(\nu)$, there exists $r>0$ such that
\begin{equation*}
 \bar{F}(D^2\phi)>\delta \quad\hbox{ in } \Omega\cap B_r(x_0)\hbox{ and }\quad \frac{\partial{\phi}}{\partial\nu} \geq \bar{\mu}(\nu)+2\delta \quad\hbox{ in }\partial\Omega\cap B_r(x_0).
\end{equation*}

Let $\nu_0$ denote the outer normal of $\partial\Omega$ at $x_0$. By adding $\delta^2(x-x_0)\cdot\nu_0 + \delta(x-x_0)^2 -M((x-x_0)\cdot\nu_0)^2$ to $\phi(x)$ and restricting the domain to $B_{r}(x_0)$ with $r\leq \delta$, we may assume that $\phi>\bar{u}$ except at $x=x_0$ in $\bar{\Omega}\cap B_r(x_0)$ and
\begin{equation}\label{condition2}
F(D^2\phi, \frac{x}{\e}) > 0 \hbox{ in } \Omega\cap B_{r}(x_0) \hbox{ and }\quad \frac{\partial{\phi}}{\partial\nu} \geq \bar{\mu}(\nu)+2\delta \hbox{ in }\partial\Omega\cap B_r(x_0).\end{equation}

\vspace{10pt}

2. Let us  choose $\alpha_0$ sufficiently small such that 

\begin{equation}\label{away}
 |\nu_x-\nu_{x_0}|^{21/20} >\e^{1-k} \hbox{ for } x\in (B_r(x_0)-B_{\e^{\alpha_0}}(x_0))\cap\partial\Omega,
\end{equation}
where $0<k<1$ is the constant given in Lemma~\ref{close}.
This is possible due to our assumption \eqref{condition}, i.e., since $\partial\Omega$ does not have any flat boundary parts.  We now perturb $\phi$ and consider $\phi^\e$ such that 

\begin{itemize}
\item[(a)] $\phi^\e$ uniformly converges to  $\phi$ as $\e\to 0$ in $\bar{\Omega}\cap B_r(x_0)$;
\item[(b)] $F(D^2\phi^\e, \frac{x}{\e})>0$  in  $B_r(x_0)$;
\item[(c)] $\partial_{\nu}\phi^\e \geq \partial_{\nu}\phi - C\e^{\alpha_0}$ on $\partial\Omega\cap B_r(x_0)$,  and  $\partial_{\nu} \phi^\e \geq 2$ on $\partial\Omega\cap B_{\e^{\alpha_0}}(x_0)$.
\end{itemize}

Note that by (b)-(c) $\phi^\e-u^\e$ can have its minimum at $y$ in $B_r(x_0)$ only if $y$ is more than $\e^{\alpha_0}$-away from $x_0$.  This step is necessary to exclude the possibility that $\nu_{x_0}$ is rational and the irrational normal direction near the minimum of $\phi^\e-u^\e$ is too close to $\nu_{x_0}$ compared to the size of $\e$ to observe the averaging behavior.

\medskip

 For convenience let $x_0=0$, $\nu_{x_0}=e_n$, and set $\gamma:=\e^{\alpha_0}$. We will construct $\phi^\e$ as $\phi + g_{\gamma}+ v_{\gamma}$, where $g_{\gamma}$ and $v_{\gamma}$ will be constructed below.
Let us first explicitly $g=g_{\gamma}$, which satisfies 
$g\in C^2 (\bar{\Omega})$, $g_{\gamma}\to 0$ as $\gamma\to 0$ in $\Omega$, and 
\begin{equation}\label{check}
\partial_\nu g\geq  -C\gamma^2 \hbox{ on } \partial\Omega, \quad \partial_{\nu} g \geq 2 \hbox{ in } B_{\gamma}(x_0)\cap \partial\Omega.
\end{equation}
Let us define $\Sigma_{\gamma}:=\{|x'|\leq 2\gamma\}\times \{-\gamma\leq x_n\leq \gamma\}$, and let us define $g$ by $g=0$ in $\Omega-\Sigma_\gamma$, and
$$
g(x',x_n) =  \frac{1}{\gamma^2}\phi(\frac{|x'|^2}{\gamma^2}) (x_n+\gamma)^3 \hbox{ in } \Sigma_{\gamma},
$$
where $\phi(r):\R\to [0,1]$ is a $C^2$ function which has support $[-2,2]$, $\phi(r)=1$ for $r\in (-1,1)$  and $\phi'(0)=0$. 

\medskip

Note that, since $\partial\Omega$ is $C^2$,  $\partial\Omega \cap\Sigma_{\gamma} \subset \{ |x_n| \leq C\gamma^2\}$: in particular, for small $\gamma$, $\Omega\cap\{x_n \geq \gamma\}$ is empty in $\Omega\cap\{|x'| \leq \gamma\}$. It follows that $g$ is in $C^2(\bar{\Omega})$.  Clearly  $g$ goes to zero in $\bar{\Omega}$ as $\gamma\to 0$, and so it remains to check \eqref{check}.
Straightforward computations yield
$$
\begin{array}{lll}
\partial_{\nu} g (x)\geq \frac{3}{\gamma^2}(x_n+\gamma)^2\phi(\frac{|x'|^2}{\gamma^2})(1-C\gamma) -  C\gamma(\frac{1}{\gamma^2}\|\phi' \|_{\infty}(\frac{2|x'|}{\gamma^2} (x_n+\gamma)^3) &\geq& 2\quad \hbox{ on } \partial\Omega\cap\{|x|\leq \gamma\}\\ \\
&\geq & -C\gamma \hbox{ in } \partial\Omega,
\end{array}
$$
from which \eqref{check} follows.

Next let $v=v_{\gamma}$ be the viscosity solution of 
$$
\left\{\begin{array}{lll}
-\mathcal{P}^+(D^2 v)  =f_\gamma:=- \mathcal{P}^+(D^2 g_{\gamma}) &\hbox{ in }& \Omega\cap B_r(x_0); \\ \\
 \partial_\nu v=0 &\hbox{ on }& \partial\Omega\cap B_r(x_0);\\ \\
  v=g_{\gamma}=0 &\hbox{ on }& \partial B_r(x_0)\cap\Omega.
  \end{array}\right.
$$

To ensure that the remaining term $v$ in the construction of $\phi^\e$ would vanish as $\e\to 0$, we estimate the $L^p$ norm of $D^2g$. Observe that, from the definition of $g$, it follows that
$$
|g_{ij}| \leq \frac{8}{\gamma}\hbox{ for } i,j=1,...,n.
$$
Hence 
\begin{equation}\label{claim}
\|g\|^p_{W^{2,p}(\Omega)} \leq (C\gamma^{-p})\gamma^n =C\gamma^{n-p},
\end{equation}
and  in particular if $p=n$,

\begin{equation}\label{claim1}
\|g\|_{W^{2,n}(\Omega)}  \leq C.
\end{equation}

Note that, due to \eqref{claim1}, $\|f_{\gamma}\|_{L^n}$ is uniformly bounded. Therefore \cite{MS} yields that  $\{v_{\gamma}\}_{\gamma>0}$ are uniformly H\"{o}lder continuous in $B_r(x_0)\cap\bar{\Omega}$ (Proposition 2.1 and Proposition 4.2 in \cite{MS}). Moreover, due to \eqref{claim},  $\|f_{\gamma}\|_{L^p}$ vanishes as $\gamma\to 0$ for $p<n$. Due to this fact and the equi-continuity of $v_{\gamma}$ one can make use of the stability of $L^p$-viscosity solutions for $n>p>n-\e_0$ for $\e_0=(n,\lambda,\Lambda)$ (see \cite{CCKS}) and employ a compactness argument as in Lemma 2.3 of \cite{S} to show that $v_{\gamma}$ uniformly converges to zero in $\bar{\Omega}\cap B_r(x_0)$ as $\gamma\to 0$. 

\medskip

Due to above properties of $g$ and $v$, $\phi^\e: = \phi+g_\sigma+v_\sigma$ with $\sigma=\e^{\alpha_0}$ satisfies (a)-(c): to check (b), note that 
$$
F(D^2\phi^\e, \frac{x}{\e})  \geq F(D^2\phi,\frac{x}{\e}) - \mathcal{P}^-(D^2g+ D^2 v) \geq F(D^2 \phi,\frac{x}{\e}) - \mathcal{P}^+(D^2g) - \mathcal{P}^+(D^2 v) >0.
$$

\medskip

4. Due to the definition of $\bar{u}$ and (a),  for $r>0$ as given above we have  $\phi^\e>u^\e$ on $\partial B_r(x_0)$ if $\e$ is sufficiently small, along a subsequence of $\e$. Moreover, $\phi^\e-u^\e$ has a local minimum at $x_\e\in B_r(x_0)\cap\bar{\Omega}$ with $x_\e\to x_0$ as $\e\to 0$: we may add a constant to $\phi^\e$ to assume that $\phi^\e(x_\e)=u^\e(x_\e)$.  Let us denote $\eta_\e = \nu_{x_\e}$. Note that $x_\e$ lies outside of $B_{\e^{\alpha_0}}(x_0)$ due to the construction of $\phi^\e$, and due to the comparison principle applied to $\phi$ and $u$,
$x_\e \in \partial\Omega \cap B_r(x_0)$.

\vspace{10pt}

5. This step is to extract only the normal component of $\phi^\e$ at $x=x_\e$ and make a new function $\varphi$ with it,
so that we do not have to worry about the tangential derivative of $\phi^\e$ at $x_\e$. We also adjust $u_\e$ accordingly. Let us decompose $\phi^\e$ into $\phi ^\e= \phi_1+\phi_2$ where
$$
\phi_1(x) = (x-x_\e)\cdot(D\phi^\e - \eta_\e (\eta_\e\cdot D\phi^\e))(x_\e).
$$
Then
\begin{equation}\label{perturb}
\eta_\e\cdot D\phi_1(x_\e) = 0, \quad D\phi_2(x_\e) = \eta_\e(\eta_\e\cdot D\phi_2)(x_\e), \quad\hbox{ and } \phi_2(x_\e)=\phi^\e(x_\e).
\end{equation}
Observe that, due to \eqref{perturb} and the fact that $\phi_1$ is a linear function, $\phi_2$ still satisfies \eqref{condition} instead of $\phi$.
Furthermore, since $\phi_2$ is smooth, we may choose $\e$
sufficiently small to replace $\phi_2$  by a linear profile
 $$
 \varphi(x):= \phi_2(x_\e) + D\phi_2(x_\e)\cdot(x-x_\e)
 $$
 such that
$$
\phi_2(x) \leq \varphi(x)+C\e^{5k/4} \hbox{ in }
B_{\e^{5k/8}}(x_\e).
$$
where $C$ depends on the $C^2$-norm of $\phi_2$.

\vspace{10pt}

Next we will tilt $u_\e$ accordingly so that we can use $\varphi$
instead of $\phi^\e$ as the test function. Let us define
$$\tilde{u}_\e(x)=u_\e(x)-\phi_1(x)-C\e^{5k/4}.$$ Then $\tilde{u}_\e$
satisfies
$$
\left\{\begin{array}{lll}
F(D^2\tilde{u}_\e, x/\e)=0 &\hbox{ in }& \Omega \\ \\
\tilde{u}_\e \leq \varphi &\hbox{ in }&\Omega\cap B_{\e^{5k/8}}(x_\e)\\ \\
|D\tilde{u}_\e \cdot \nu_x - g(\frac{x}{\e})| \leq C\e^{5k/4}
&\hbox{ on }& \partial\Omega\cap B_{\e^{5k/8}}(x_\e).
\end{array}\right.
$$

Note that
\begin{equation}\label{almost_contact}
\tilde{u}_\e(x_\e) = \varphi(x_\e)-C\e^{5k/4}.
\end{equation}
\vspace{10pt}

6. Next we will approximate $\Omega$ with a strip domain with the help of Lemma~\ref{close}. For $\e$ and $x_\e$ as given above, let us pick $p_\e\in\partial\Omega$ such that
\begin{itemize}
\item[(a)] $|p_\e-x_\e| \leq \e$,\\
\item[(b)] At $p_\e$,  $\partial\Omega$ is normal to $\nu_\e$: an irrational vector.
\end{itemize}

Let $v_\e$ solve the problem
$$
\left\{\begin{array}{lll}
F(D^2 v_\e,x/\e)=0 &\hbox{ in }& \Sigma_\e:=\{-\e^k \leq (x-p_\e)\cdot\nu_\e \leq 0\}\\ \\
\frac{\partial v_\e}{\partial\nu} =g(\frac{x}{\e}) +C\e^{5k/4} &\hbox{ on }& \{(x-p_\e)\cdot\nu_\e=0\}\\ \\
v_\e=\varphi &\hbox{ on }& \{(x-p_\e)\cdot\nu_\e = -\e^k\}.
\end{array}\right.
$$

Due to Lemma~\ref{close}, we have $|\omega_\e-v_\e| \leq \e^{k+a}$ in $\Omega\cap\Sigma_\e\cap B_{\e^{2k/3}}(x^\e)$, where $\omega_\e$ solves
$$
\left\{\begin{array}{lll}
F(D^2\omega_\e, \frac{x}{\e}) = 0 &\hbox{ in }& \Omega\cap\{-\e^k \leq (x-x_\e)\cdot \nu_\e\}\cap B_{\e^{5k/8}}(x_\e),\\ \\
\frac{\partial \omega_\e}{\partial\nu} =g(\frac{x}{\e}) +C\e^{5k/4} &\hbox{ on }& \partial\Omega \cap \{-\e^k \leq (x-p_\e)\cdot\nu_\e\}\cap B_{\e^{5k/8}}(x_\e),\\ \\
\omega_\e=\varphi &\hbox{ on }& \{ (x-p_\e)\cdot\nu_\e=-\e^k\} \cup \partial B_{\e^{5k/8}}(x_\e).
\end{array}\right.
$$

Now comparison principle applied to $\tilde{u}^\e$ and $\omega_\e$ in $\{ -\e^k \leq (x-x_\e)\cdot\nu\} \cap\Omega\cap B_{\e^{5k/8}}(x_\e)$ yields that 

$$
\tilde{u}_\e \leq w_\e \leq v_\e + \e^{k+a} \hbox{ in }
\Omega\cap\Sigma_\e \cap B_{\e^{2k/3}}(x_\e).
$$
Hence we have
\begin{equation}\label{difference}
\tilde{u}_\e(x_\e) \leq  v_\e(x_\e) + \e^{k+a}.
\end{equation}
Combining \eqref{almost_contact} and \eqref{difference},
\begin{equation} \label{fin}
\varphi(x_\e) \leq v_\e(x_\e) + \e^{k+a} + C \e^{5k/4}.
\end{equation}

7. Lastly we will use Theorem~\ref{continuity} to derive a contradiction to \eqref{fin}. Let us choose $y_\e\in \e\Z^n, |x_\e-y_\e| \leq \e$ and define
$$
\tilde{v}_\e(x):= \e^{-k} v_\e(\e^k (x+y_\e))
$$
so that $\tilde{v}_\e$ solves $(P_{\e^{1-k}}^{\nu_\e})$ in $\Pi_{\nu^\e}(z_\e)$ with $|z_\e| \leq \e^{1-k}$.
Due to \eqref{away} and 
Theorem~\ref{continuity}, we can choose $\e$ is sufficiently small
so that
$$
|\mu(\tilde{v}_\e)-\bar{\mu}(\nu_{x_0}) | <\frac{\delta}{2}.
$$

Then due to \eqref{condition2}, the comparison principle  applies  to $\varphi (x+y_\e)-\frac{\delta}{2}((x-z_\e)\cdot\nu_\e+1)$ and $\tilde{v}_\e$ to yield that
$$
\tilde{v}_\e (0) \leq \e^{-k}\varphi (y_\e) -\frac{\delta}{2}.
$$
After scaling back to $v_\e$, we get
$$
v_\e(y_\e) \leq \varphi(y_\e) - \frac{\delta}{2} \e^k
$$
Since $|v_\e(x_\e)-v_\e(y_\e)| \leq C\e^{\alpha}$ for any $\alpha$ due to Theorem~\ref{thm:reg2}, we choose $\alpha>k$ to obtain
$$
v_\e(x_\e) \leq \varphi(x_\e) - \frac{\delta}{4}\e^k
$$ 
if $\e$ is sufficiently small: this yields a contradiction to \eqref{fin}, and we can conclude.

\end{proof}

\medskip

\textbf{ Proof of Theorem~\ref{main2}} 

Due to Proposition~\ref{prop:main} and the comparison principle (Theorem~\ref{thm:comp}), we have $\bar{u} \leq \underline{u}$. But due to definition $\bar{u} \geq \underline{u}$. Therefore we conclude that $\bar{u}=\underline{u}$: this is equivalent to the local uniform convergence of the sequence $u_\e$ to $u=\bar{u}=\underline{u}$, which is the unique viscosity solution of $(\bar{P})$.

\hfill$\Box$

\end{document}